\documentclass[11pt,twoside,a4paper]{article}

\usepackage{amsmath, amsthm}
\usepackage{amscd}
\usepackage{amssymb, amsfonts}
\usepackage{array}
\usepackage{color}
\usepackage[all]{xy}

\newcommand{\alert}[1]{{\color{red}#1}}

\newtheorem{thm}{Theorem}[section]
\newtheorem{theorem}[thm]{Theorem}

\newtheorem{lemma}[thm]{Lemma}
\newtheorem{proposition}[thm]{Proposition}

\theoremstyle{definition}

\newtheorem{free text}[thm]{}
\newtheorem{definition}[thm]{Definition}
\newtheorem{example}[thm]{Example}
\newtheorem{examples}[thm]{Examples}
\newtheorem{observation}[thm]{Observation}
\newtheorem{observations}[thm]{Observations}
\newtheorem{remark}[thm]{Remark}
\newtheorem{remarks}[thm]{Remarks}

\newcommand{\N} {\mathbb{N}}
\newcommand{\Z} {\mathbb{Z}}
\newcommand{\R} {\mathbb{R}}
\newcommand{\C} {\mathbb{C}}
\newcommand{\A} {\mathbb{A}}
\newcommand{\bP} {\mathbf{P}}
\newcommand{\bK} {\mathbf{K}}

\newcommand{\bk} {\Bbbk}
\newcommand{\cA} {\mathcal{A}}

\newcommand{\cF} {\mathcal{F}}
\newcommand{\cH} {\mathcal{H}}
\newcommand{\cJ} {\mathcal{J}}
\newcommand{\cK} {\mathcal{K}}
\newcommand{\cL} {\mathcal{L}}

\newcommand{\cO} {\mathcal{O}}
\newcommand{\cX} {\mathcal{X}}
\newcommand{\cY} {\mathcal{Y}}
\newcommand{\cB} {\mathcal{B}}

\newcommand{\fg} {\mathfrak{g}}
\newcommand{\fh} {\mathfrak{h}}

\newcommand{\bU}{\mathbf{U}}
\newcommand{\rGL}{\mathrm{GL}}
\newcommand{\bGL}{\mathbf{GL}}

\newcommand{\bG} {\mathbf{G}}

\newcommand{\zero} {{\bar{0}}}
\newcommand{\uno} {{\bar{1}}}
\newcommand{\one} {{\bar{1}}}

\newcommand{\salggrc} {\mathrm{(salg)^{\mathrm{gr}}_\C}}
\newcommand{\smod} {\mathrm{(smod)}}

\newcommand{\smodstc} {\mathrm{(smod)^{\mathrm{st}}_\C}}
\newcommand{\smodgrc} {\mathrm{(smod)^{\mathrm{gr}}_\C}}
\newcommand{\salg} {\mathrm{(salg)}}

\newcommand{\grps} {\mathrm{(grps)}}

\newcommand{\sgrps} {\mathrm{(sgrps)}}
\newcommand{\sHCp} {\mathrm{(sHCp)}}
\newcommand{\Lie} {\mathrm{Lie}}
\newcommand{\ad} {\mathrm{ad}}
\newcommand{\gr} {\mathrm{gr}}
\newcommand{\st} {\mathrm{st}}
\newcommand{\id} {\mathrm{id}}
\newcommand{\sLie} {\mathrm{(sLie)}}

\newcommand{\sLiestc} {\mathrm{(sLie)^{\mathrm{st}}_\C}}
\newcommand{\sLiegrc} {\mathrm{(sLie)^{\mathrm{gr}}_\C}}
\newcommand{\Hom}{{\mathrm{Hom}\,}}

\newcommand{\lra}{\longrightarrow}
\newcommand{\aut}{\mathrm{aut}}

\newcommand{\baut}{\mathrm{\overline{aut}}}

\newcommand{\beq}{\begin{equation}}
\newcommand{\eeq}{\end{equation}}

\newcommand{\fk}{\mathfrak {k}}

\newcommand{\fsl}{\mathfrak {sl}}

\newcommand{\fp}{\mathfrak {p}}

\newcommand{\fu}{\mathfrak {u}}

\newcommand{\fgl}{\mathfrak {gl}}

\newcommand{\bbar}{\overline{b}}

\newcommand{\wbar}{\overline{w}}
\newcommand{\xbar}{\overline{x}}
\newcommand{\ybar}{\overline{y}}
\newcommand{\zbar}{\overline{z}}
\newcommand{\zetabar}{\overline{\zeta}}

\newcommand{\Abar}{\overline{A}}
\newcommand{\Bbar}{\overline{B}}
\newcommand{\Cbar}{\overline{C}}

\newcommand{\re}{\mathrm{re}}

\newcommand{\be}{\beta}

\begin{document}

{\ }

\vskip-51pt

   \centerline{\small  \textsl{Communications in Mathematical Physics\/}  \textbf{397}
(2023), no.\ 2, 937--965   \ \ \ --- \ \ \  \textbf{DOI:}  10.1007/s00220-022-04502-x}
 \vskip3pt
   \centerline{\small --- preprint  {\tt arXiv:2003.10535} [math.RA] (2020) --- }
 \vskip1pt
   \centerline{\small  \textsl{The original publication is available at\/}
\  \texttt{https://link.springer.com/article/10.1007/s00220-022-04502-x}}

\vskip27pt   {\ }

\centerline{\Large \bf REAL FORMS of COMPLEX LIE}
 \vskip11pt
\centerline{\Large \bf SUPERALGEBRAS and SUPERGROUPS}

\vskip45pt

\centerline{Rita FIORESI$^\flat$ \ , \ \  Fabio GAVARINI$^\#$}

\bigskip

\centerline{\it $^\flat$ Dipartimento di Matematica, Universit\`a
di Bologna } \centerline{\it piazza di Porta San Donato, 5  ---
I-40127 Bologna, Italy} \centerline{{\footnotesize e-mail:
rita.fioresi@unibo.it}}

\bigskip

\centerline{\it $^\#$ Dipartimento di Matematica, Universit\`a di
Roma ``Tor Vergata'' } \centerline{\it via della ricerca
scientifica 1  --- I-00133 Roma, Italy}

\centerline{{\footnotesize e-mail: gavarini@mat.uniroma2.it}}

\vskip35pt

\begin{abstract}
   We investigate the notion of  {\sl real form\/}  of complex Lie superalgebras and supergroups, both in the  {\sl standard\/} and {\sl graded\/}  version.  Our functorial approach allows most naturally to go from the superalgebra to the supergroup and retrieve the real forms as fixed points, as in the ordinary setting.  We also introduce a more general notion of compact real form for Lie superalgebras and supergroups, and we prove some existence results for Lie superalgebras that are simple contragredient and their associated connected simply connected supergroups.
\end{abstract}

\vskip35pt
%
%

\section{Introduction}   \label{intro}

   The study of real forms of complex contragredient Lie superalgebras was initiated by V.\ G.\ Kac in his foundational work  \cite{kac}  and then carried out by M.\ Parker in  \cite{parker}
and  V.\ Serganova in  \cite{serganova},  where also symmetric superspaces were accounted for. Later on, Chuah in  \cite{chuah-mz}  gave another thorough classification of such real forms using Vogan diagrams and Cartan automorphisms.  In fact, as it happens for the ordinary setting, we have a one to one correspondence between real structures on a contragredient Lie superalgebra  $ \fg \, $,  and its Cartan automorphisms  $ \aut_{2,4}(\fg) \, $,  that is automorphisms that are involutions on the even part and whose square is the identity on the odd part of  $ \fg \, $.  This translates to a bijection between the antilinear involutions  $ \baut_{2,2}(\fg) $  of  $ \fg $  and the linear automorphisms  $ \aut_{2,4}(\fg) \, $.  In the ordinary setting, that is for $ \, \fg = \fg_\zero \, $,  this correspondence is explicitly obtained through the Cartan antiinvolution  $ \omega_\zero \, $,  whose fixed points give the compact form of  $ \fg_\zero \, $.  In the supersetting, as we shall see, such antiinvolution is replaced by an antilinear automorphism  $ \, \omega \in \baut_{2,4}(\fg) \, $.  This prompts for a more general treatment of real structures and real forms of superspaces and superalgebras, together with their global versions, where we consider both cases  $ \, \baut_{2,s}(\fg) \, $  and  $ \, \aut_{2,s}(\fg) \, $,  for $ \, s = 2, 4 \, $.  We shall refer to such real structures and real forms as  {\sl standard\/}  and  {\sl graded\/};  they were introduced in  \cite{pe},  \cite{serganova}.

\medskip

   The paper is organized as follows.  Sec.\ 2 contains preliminaries that help to establish our notation.  In Sec.\ 3, we begin by defining the notion of standard and graded real structure on a superspace  $ V $  as a pair  $ (V,\phi) $  with  $ \, \phi \in \baut_{2,2}(\fg) \, $  or  $ \, \baut_{2,4}(\fg) \, $,  respectively.  We obtain two categories,  $ \smod_\C^\st $  and  $ \, \smod_\C^\gr \, $,  that we compactly denote  $ \smod_\C^\bullet \, $  whenever there is no need to remark the difference; similarly, we define the corresponding categories of superalgebras  $ \salg_\C^\bullet \, $.  As expected, given a real structure, the associated real form is given by the fixed points of the antiautomorphism, however in the graded case, the functorial point of view is most fruitful, because such points cannot be seen over the complex field.  After establishing the terminology and definitions, we can then give naturally the notion of real structures and real forms of Lie superalgebras, following and extending the work  \cite{pe}. These real structures and real forms do integrate: thus, in Sec.\ 4, we obtain the category of complex supergroups with standard or graded real structures, that we denote with  $ \sgrps_\C^\st $  and  $ \sgrps_\C^\gr \, $,  or more compactly  $ \sgrps_\C^\bullet \, $.  We also briefly discuss the super Harish-Chandra pairs (sHCp) approach in this context (see
also  \cite{cfk1, ga1, mas}).  Our main result for this part is the following
(see Theorem \ref{real-form_G}).
\medskip

%
%
 {\bf Theorem A.}
  {\it If  $ \, \big(\bG,\Phi\big) \in \sgrps_\C^\bullet \, $,  the real form  $ \, \bG^\Phi $  of  $ \, \bG \, $,  given by the fixed points of  $ \, \Phi \, $,  is
%
%
  $$  \bG^\Phi(A)  \; = \;  \Big( G_+^{\Phi_+} \! \times \mathbb{A}_{\bullet,\C}^{0\,|d_1} \Big)(A) \;\;\; ,   \qquad   \forall \;\; A \in \salg_\C^\bullet  $$
%
%
 where  $ \, G_+^{\Phi_+} $  is the ordinary underlying real form of  $ G_+ $  and  $ \, \mathbb{A}_{\bullet,\C}^{0\,|d_1} \, $  is a real form of a purely odd affine superspace.  In particular, the supergroup functor  $ \, \bG^\Phi $  is representable.}
%
%

\medskip

   In the remaining part of the paper, we discuss compact real forms of contragredient complex Lie superalgebras and the corresponding supergroups, using the results detailed above.
                                                  \par
   In the ordinary setting, a real Lie algebra is compact if it is embedded into some orthogonal or equivalently unitary Lie algebra.  For a Lie superalgebra  $ \fg \, $,  many authors (see \cite{chuah-mz},  \cite{cfv},  \cite{cf})  replace this notion with the requirement that $ \, \fg = \fg_0 \, $  and the latter compact.  We take a more general approach, allowing  $ \fg $  to have odd elements.  For this reason, in Sec.\ 5, we need to examine super Hermitian forms, in the standard and graded context, and the corresponding unitary Lie superalgebras.  In our Subsec.\ 5.4, we retrieve in our language the physicists' definition of unitary Lie superalgebra (see  \cite{vsv}  and references therein), but also a graded version of it, obtained as fixed points of the  {\sl superadjoint}   --- that is, the supertranspose complex conjugate.  We regard this example very significant and natural, since it is obtained via an antilinear morphism in  $ \, \baut_{2,4}(\fgl(m|n)) \, $,  which has a categorical motivation  (see \cite{fl},  Ch.\ 1, and also  \cite{serganova, pe}).
                                                  \par
   In Sec.\ 6, we formulate our notion of  {\sl compact\/}  Lie superalgebra as one admitting an embedding into a unitary Lie superalgebra for a suitable positive definite super Hermitian form. We shall call this {\it super-compact}.  Then, we are finally able to introduce  $ \, \omega \in \baut_{2,4}(\fg) \, $,  generalizing the Cartan antiinvolution  $ \omega_\zero $  mentioned above, and to prove the correspondence between  $ \, \baut_{2,4}(\fg) \, $  and  $ \, \aut_{2,2}(\fg) \, $  and between  $ \, \baut_{2,2}(\fg) \, $  and  $ \, \aut_{2,4}(\fg) \, $.  Our main result for this part is the following (see  Theorems \ref{cpt-form-gr}  and  \ref{cpt-form-st}):

\medskip

   {\bf Theorem B.}  {\it Let  $ \fg $  be a simple complex contragredient Lie superalgebra.  Then:
 \vskip5pt
   \quad   (a)\;  $ \fg $  admits a  {\sl graded},  {\sl super-compact}  real form, given via  $ \, \omega \in \baut_{2,4}(\fg) \, $;
 \vskip3pt
   \quad   (b)\;  if\/  $ \fg $  is of type 1, then  $ \fg $  admits a  {\sl standard},  {\sl compact} real form;
 \vskip3pt
   \quad   (c)\;  if\/  $ \fg $  is of type 2, then  $ \fg $  has no  {\sl standard},  {\sl compact} real form.
 \vskip5pt
 In all cases, such super-compact or compact forms are unique up to inner automorphisms.}

\medskip

   We end our treatment giving a global version of the previous results
(see Theorems \ref{cpt-form-gr_sgrps},  \ref{cpt-form-st_sgrps}).

\medskip

{\bf Theorem C.}  {\it Let  $ \, \bG $  be a complex supergroup with  $ \, \fg = \Lie(G) \, $  being simple contragredient.  Then  $ \, \bG $  admits a  {\sl graded},  {\sl super-compact}  real form, which is unique up to inner automorphisms.
                                                         \par
   If  $ \, \fg $  is of type 1, then  $ \, \bG $  admits a  {\sl standard},  {\sl compact}  real form, unique up to inner automorphisms.  If  $ \, \fg $  is of type 2, then  $ \, \bG $  has no  {\sl standard},  {\sl compact}  real form.}

\vskip11pt

   {\bf Acknowledgements}.  This work was partially supported by the MIUR  {\sl Excellence Department Project\/}  awarded to the Department of Mathematics of the University of Rome ``Tor Vergata'', CUP E83C18000100006.  The authors thank M.-K.\ Chuah for helpful comments.
                                                 \par
   R.\ Fioresi and F.\ Gavarini thank respectively the department of Mathematics of Rome ``Tor Vergata'' and of Bologna for the wonderful hospitality while this work was prepared.

\vskip27pt

\section{Real structures of superspaces and superalgebras}

 Let our ground field  $ \, \bk = \C \, $.  For notation and basic facts on supergeometry, see  \cite{ccf, vsv, dm, ma}.

\subsection{Real structures of super vector spaces}

\begin{definition}  \label{super-str1}
 Let  $ \, V \,$  be a complex super vector space.  We call  {\it (generalized) real structure}, of  {\sl standard\/}  or  {\sl graded\/}  type respectively, on  $ V $  any  $ \C $--antilinear  super vector space morphism  $ \, \phi : V \longrightarrow V \, $  such that
 \vskip3pt
   \qquad   {\it (1)}\, \qquad  $ \phi^2\big|_{V_\zero} = \, \text{\rm id}_{V_\zero} \;\; $,
                                                        \par
   \qquad \quad  {\it (2.s)}\, \quad  $ \phi^2\big|_{V_\uno} = \, +\text{\rm id}_{V_\uno} \;\; $  \quad  ({\it {\sl standard}  real structure}),
                                                        \par
   \qquad \quad  {\it (2.g)}\, \quad  $ \phi^2\big|_{V_\uno} = \, -\text{\rm id}_{V_\uno} \;\; $  \quad  ({\it {\sl graded}  real structure}).
 \vskip3pt
   Note that giving a real structure on  $ V $  is the same as giving on it a  $ \C $--antilinear  action of the cyclic group  $ \Z_4 $  which on the even part  $ V_\zero $  factors through its quotient  $ \Z_2 \, $.  The action factors through  $ \Z_2 $  on all of  $ V $  if and only if the corresponding real structure is standard.
 \vskip3pt
   We call the subspace  $ V^\phi $  of fixed points  {\it standard\/}  or  {\it graded real form\/}  of  $ V \, $.  This  $ V^\phi $  is a real form of  $ V $  in the usual sense if and only if  $ \phi $  is an involution, i.e.\ in the standard case.
                                                         \par
   If in addition  $ V $  is a  Lie superalgebra, we require  $ \phi $  to be a Lie superalgebra (anti-linear) morphism, i.e.\ to preserve the Lie (super)bracket of  $ V $.  Similarly, we require the
analogous property when  $ \phi $  is an associative superalgebra, a superbialgebra, a Hopf superalgebra, etc.
\end{definition}

\smallskip

\begin{definition}
 We  define the categories  $ \smod^\st_\C $ and $ \smod^\gr_\C $  of  $ \C $--supermodules  with standard or graded real structure as follows.  The objects are pairs  $ (V,\phi) $,  where  $ V $  is any  $ \C $--supermodule  with  $ \phi $  as its real structure (standard or graded).  The morphisms from an object  $ \big(V',\phi'\big) $  to an object  $ \big(V'',\phi''\big) $   --- both either standard or graded ---   are those morphisms of  $ \C $--supermodules  $ \, f : V' \longrightarrow V'' \, $  such that  $ \, f \circ \phi' = \phi'' \circ f \, $;  in short, any such  $ f $ preserves the  $ \Z_4 $--action.  We use notation  $ \smod^\bullet_\C $  to denote either one of these categories, with  $ \, \bullet \in \{\mathrm{st},\mathrm{gr}\} \, $.
                                                                        \par
  If  $ \, (V,\phi) \in \smod^\bullet_\C \, $  and  $ \, V' \subseteq V \, $  is a super vector subspace of  $ V \, $,  with  $ \, \phi(V') = V' \, $,  we say that the real structure  $ (V',\phi|_{V'}) $  is induced by  $ (V,\phi) $  and we write  $ \, (V',\phi|_{V'}) \subseteq (V,\phi) \, $.
                                                            \par
   We can similarly define the categories  $ \salg^\st $  and $ \salg^\gr $  of all unital associative commutative superalgebras with a standard or graded real structure and the categories  $ \sLie^\st $ and  $ \sLie^\gr $ of all Lie  $ \C $--superalgebras  with a standard or graded real structure.
                                                                        \par
   As customary with superalgebras   --- cf.\ \cite{pe}  ---   for  $ \, A \in \salg_\C^\bullet \, $  we denote the real structure with the notation  $ \, a \mapsto \widetilde{a} \, $ ,  and we call such map  {\it standard\/  {\rm or}  graded conjugation}.
\end{definition}

\vskip5pt

\begin{remark}
 By its very construction,  $ \smod_\C^\bullet $  is a subcategory of the category  $ \smod_\C^{\Z_4} $  of supervector spaces with a $ \Z_4 $--action.  Moreover, the latter is also a tensor category, and then  $ \smod_\C^\bullet $  is actually a  {\sl tensor subcategory\/}:  namely, if  $ \, \big(V',\phi'\big) , \big(V'',\phi''\big) \in \smod_\C^\bullet \, $,  then  $ \, \phi' \otimes \phi'' \, $  is a real structure   --- of the correct type, i.e.\ either standard or graded ---   on  $ \, V' \otimes V'' \, $.
\end{remark}

\subsection{Real structures and real forms of functors}

   We now want to express functorially the notion of (generalized) real structure described in the previous section.  Assume that  $ V $  is a complex super vector space and consider the functor
  $$  h_V : \salg_\C \!\relbar\joinrel\relbar\joinrel\relbar\joinrel\longrightarrow
\smod_\C  \quad ,   \;\;\;\quad  A \mapsto {\big( A \otimes V \big)}_\zero \, =
\, A_\zero \otimes V_\zero + A_\uno \otimes V_\uno
$$
(the definition on the morphisms being clear), the  $ \Z_2 $--grading  being given by
$ \, {\big(h_V(A)\big)}_{\overline{z}} := A_{\overline{z}} \otimes V_{\overline{z}} \, $  for each  $ \, \overline{z} \in \Z_2 \, $.  This in fact is identified with the functor of points of the affine superspace  $ \A(V) $   --- see \cite{ccf},  Ch.\ 10.  When in addition  $ \, V = \fg \in \sLie_\C \, $  is a complex Lie superalgebra, the associated functor  $ \, h_\fg \, $  is actually valued in the category  $ {(\Z_2\text{--Lie})}_\C $  of complex,  $ \Z_2 $--graded Lie algebras, i.e.\  it is a functor  $ \; h_\fg : \salg_\C \!\relbar\joinrel\relbar\joinrel\relbar\joinrel\longrightarrow {(\Z_2\text{--Lie})}_\C \; $.

\smallskip

\begin{definition}  \label{gr-str2}
 Let  $ V $  a complex super vector space.
%
%
   For  $ \, \bullet \in \{\mathrm{st},\mathrm{gr}\} \, $,  let  $ \; \mathcal{L}_V \! := \mathcal{R} \circ h_V \circ \mathcal{F} \; $  where  $ \, \mathcal{F} : \salg^\bullet_\C \!\longrightarrow\! \salg_\C \, $  is the obvious forgetful functor and  $ \, \mathcal{R} : \smod_\C \!\longrightarrow\! \smod_\R \, $  is the obvious functor of scalar restriction.
 We call  {\it real structure} on  $ \mathcal{L}_V $  any natural transforma-\break{}tion $ \; \varphi : \mathcal{L}_{V} \!\relbar\joinrel\longrightarrow
\mathcal{L}_{V} \; $  such that for each  $ \, A \in \salg^\bullet_\C \, $  the map  $ \; \varphi_A : \mathcal{L}_{V}(A) \relbar\joinrel\longrightarrow \mathcal{L}_{V}(A) \; $  is
 \vskip2pt
   {\it (1)}\,  conjugate  $ A_\zero $--linear,  i.e.\  $ \;\; \varphi_A(a_1 X_1 + a_2 X_2) = \widetilde{a}_1 \, \varphi_A(X_1) + \widetilde{a}_2 \, \varphi_A(X_2) \; $  for all  $ \, a_i \in A_\zero \, $,  $ \, X_i \in \mathcal{L}_V(A) \, $,
 \vskip1pt
   {\it (2)}\,  parity-preserving, i.e.\  $ \;\; \varphi_A\big( A_{\overline{a}} \otimes V_{\overline{a}} \big) \, \subseteq \, A_{\overline{a}} \otimes V_{\overline{a}} \;\; $,
 \vskip1pt
   {\it (3)}\,  involutive, i.e.\  $ \, \varphi_{\!A}^{\;\,2} = 1 \;\; $.
%
%
 \vskip1pt
\noindent
 Such a  $ \varphi $  is called  {\sl standard},  resp.\  {\sl graded},  if  $ \, \bullet = \mathrm{st} \, $,  resp.\  $ \, \bullet = \mathrm{gr} \, $.
                                              \par
   If in addition  $ \, V = \fg \in \sLie_\C$ is a Lie superalgebra, we define a  {\it real structure\/}  on  $ \mathcal{L}_{\fg} $  as above, but adding the further condition that each  $ \varphi_A $  be a morphism of  ($ \Z_2 $--graded)  Lie algebras, i.e.\
$$
 \varphi_A\big(\big[X_1,X_2\big]\big)  \, = \, \big[\varphi_A(X_1),\varphi_A(X_2)\big]   \hskip27pt  \forall \;\; X_1, X_2 \in \mathcal{L}_{\fg}(A)
$$
\end{definition}

\smallskip

\begin{theorem}\label{eq-real}
   For every  $ \, \fg \in \sLie_\C $,  there exists a canonical, bijection between standard, resp.\ graded, real structures on  $ \mathcal{L}_{\fg} $  and standard, resp.\ graded, real structures on  $ \fg \, $.  Furthermore, this bijection induces an equivalence between the category of functors  $ \mathcal{L}_{\fg} $  with standard, resp.\ graded, real structures and  $ \sLiestc \, $,  resp.\  $ \sLiegrc $   --- and similarly for  $ \mathcal{L}_V $  and  $ \smodstc \, $,  resp.\  $ \smodgrc \, $.
\end{theorem}

\begin{proof}  If  $ \varphi $  is a real structure on  $ \cL_V $,  we have a corresponding real structure  $ \; \phi : V \rightarrow V \, $  on  $ V $  defi-\break{}ned by  $ \, \phi(v) := \varphi_\C(v) \, $.  Conversely, if  $ \phi $  is a real structure on  $ V $,  for each  $ \, A \in \salg_\C^\bullet \, $  we define a corresponding  $ \varphi_A $  by  $ \, \varphi_A(a \otimes v) := \widetilde{a} \otimes \phi(v) \, $.   Details can be found in  \cite{pe}, Theorem 2.6.
\end{proof}

\smallskip

   We now turn to examine generalized real forms in the functorial language.

\smallskip

\begin{definition}  \label{realf-def}
 Let  $ V $  be a complex super vector space with real structure  $ \phi \, $,  and  $ \varphi $  the corresponding real structure on the functor  $ \cL_V \, $,  as in  Theorem \ref{eq-real}.  We define  {\it real form (standard or graded) of\/}  $ \cL_V $  as being the functor  $ \; \cL_V^{\,\varphi} : \salg^\bullet_\C \!\lra \smod_\R \; $  given on objects by
  $$  \cL_V^{\,\varphi}(A)  \, := \,  {\cL_V(A)}^{\varphi_A}  \, = \,  \big\{\, x \in \mathcal{L}_{V}(A) \,\big|\, \varphi_A(x) = x \,\big\}   \eqno  \forall \;\;  A \in \salg^\bullet_\C  \quad  $$
 --- in other words,  $ \cL_V^{\,\varphi}(A) $  is the submodule of  $ \varphi_A $--invariants, i.e.\  the fixed points of  $ \varphi_A \, $,  in  $ \cL_V(A) \, $  --- and in the obvious way on morphisms.
 {\sl Note\/}  also that if  $ \, V = \fg \in \sLie_\C \, $  is in fact a complex Lie superalgebra, and  $ \phi $  is a real structure in the Lie sense, then each  $ \, \cL_\fg^{\,\varphi}(A) \, $  is automatically a  $ \Z_2 $--graded  {\sl real\/}  Lie subalgebra in  $  \cL_\fg(A) \, $,  so that  $ \cL_\fg^{\,\varphi} $  is actually a functor from  $ \salg^\bullet_\C $  to  $ {(\Z_2\text{--Lie})}_\R \, $,  the category of  $ \Z_2 $--graded  {\sl real\/}  Lie algebras.
\end{definition}

\smallskip

\begin{proposition}  \label{gr-rep1}
  With notation as above, assume  $ V $  is finite dimensional.  Then the functor  $ \mathcal{L}^{\,\varphi}_V $ is representable and
  it is represented by the symmetric superalgebra
$ \, S(V^*) \in \salg^\bullet_\C \, $.
\end{proposition}

\noindent
    {\it Proof}.  It is immediate by the following chain of equalities:
  $$  \begin{array}{rl}
   \mathcal{L}_V^{\,\varphi}(A)  &  \! = \;  {(A \otimes V\,)}_\zero^{\varphi_A}  \; = \;  {(A \otimes V\,)}_\zero^{\,\Z_4}  \; = \;  {\Hom\big( V^* , A \big)}^{\,\Z_4}  \; =  \\
   &  \! = \;  {\Big( \Hom_{\salg_\C}\big( S(V^*) \, , A \big) \!\Big)}^{\Z_4^{\phantom{\big|}}}  \, = \;  \Hom_{\salg^\bullet_\C}\big( S(V^*) \, , A \big)   \qquad \qquad \hskip4pt  \square
\end{array}  $$

\smallskip

\begin{remarks}  \label{altern_defs_real-strct's}
 The following are alternative, equivalent ways to introduce the notion of ``real structure'' on the functor  $ \cL_V $  for any  $ \, V \in \smod_\C \, $:
 \vskip1pt
   \textit{(a)}\;  Let  $ \overline{V} $  be the complex-conjugate of  $ V $,  that is  $ V $  itself as real vector space endowed with the conjugate complex structure.  Let  $ \, \mathcal{F} : \salg^\bullet_\C \!\longrightarrow\! \salg_\C \, $  be the forgetful functor considered above, and  $ \, \mathcal{C} : \salg^\bullet_\C \!\longrightarrow\! \salg^\bullet_\C \, $  be the functor given on objects by  $ \, \mathcal{C}(A) := \overline{A} \, $  and on morphisms by  $ \, \mathcal{C}(f) := f \, $;  \,then, setting  $ \, \cL'_V := h_V \circ \mathcal{F} \, $  and  $ \, \cL'_{\overline{V}} := h_{\overline{V}} \circ \mathcal{F} \, $  we have  $ \, h_{\overline{V}} = h_V \circ \mathcal{C} \, $  and  $ \, \cL'_{\overline{V}} = \cL'_V \circ \mathcal{C} \, $.  Using this language, giving a real structure on  $ \cL_V $  is equivalent to giving a pair of natural transformations  $ \; \varphi'_+ : \cL'_V \relbar\joinrel\longrightarrow \cL'_{\overline{V}} \; $  and  $ \; \varphi'_- : \cL'_{\overline{V}} \relbar\joinrel\longrightarrow \cL'_V \; $  that are parity preserving and such that  $ \; \varphi'_- \circ \varphi'_+ = \id_{\cL'_V} \; $  and  $ \; \varphi'_+ \circ \varphi'_- = \id_{\cL'_{\overline{V}}} \; $.
 \vskip1pt
   \textit{(b)}\;  If  $ \varphi $  is a real structure on  $ \cL_V \, $,  \,then  $ \; \varphi_A : \cL_V(A) \relbar\joinrel\longrightarrow \cL_V(A) \; $   --- for each  $ \, A \in \salg_\C^\bullet \, $  ---   is a real structure,  \textsl{in classical sense},  on the  $ \Z_2 $--graded  complex vector space  $ \cL_V(A) \, $,  which is conjugate  $ A_\zero $--linear  and preserves the  $ \Z_2 $--grading.  Now, let  $ {(\Z_2\textrm{--mod})}_\C^{\textrm{rs}} $  be the category of  $ \Z_2 $--graded  complex vector spaces  \textsl{with a conjugate  $ A_\zero $--linear,  $ \Z_2 $--graded  real structure},  and  $ \, {(\Z_2\textrm{--mod})}_\C^{\textrm{rs}} \,{\buildrel {\cF_*} \over {\relbar\joinrel\longrightarrow}}\, {(\Z_2\textrm{--mod})}_\C \; $  the obvious forgetful functor.  Then, just rephrasing the  Definition \ref{gr-str2},  we can quickly find that giving a real structure on  $ \cL_V $  is actually equivalent to giving a functor
 $ \; \dot{\cL} : \salg_\C^\bullet \relbar\joinrel\relbar\joinrel\longrightarrow {(\Z_2\textrm{--mod})}_\C^{\textrm{rs}} \; $
 such that  $ \; \cF_* \circ \dot{\cL} = \cL_V \circ \cF \; $.
                                                   \par
   Indeed, roughly speaking the condition  $ \; \cF_* \circ \dot{\cL} = \cL_V \circ \cF \; $  means that  ``$ \dot{\cL} $  coincides with  $ \cL_V $  up to forgetting any real structure'', hence we can say that, in a nutshell, any such functor  $ \dot{\cL} $  is (sort of)  ``$ \cL_V $  endowed with a pointwise real structure''.
\end{remarks}

\medskip

 \subsection{Real affine superspaces}  \label{sec_affine-real-spaces}  {\ }
 \vskip-3pt
   Let $ V $  be a complex super vector space of finite (super) dimension  $ r|s \, $;  its associated affine superspace  $ \A(V) $  is the complex superspace described by the functor $ \, \cL_V : \salg_\C \!\lra \smod_\C \, $,  which is represented by the complex commutative superalgebra  $ S(V^*) \, $.  If in addition  $ \phi $  is a real structure on  $ V $,  we  {\sl define\/}  the  {\it real affine superspace  $ \A(V,\phi) $  associated to  $ (V,\phi) $}
 as the ``superspace with real structure'' whose functor of points is  $ \cL_V^\varphi $  (as in  Proposition \ref{gr-rep1}  above), represented by the symmetric superalgebra  $ S(V^*) $  with real structure canonically induced by that of  $ V $.  We will also write  $ \, \A_{\bullet,\C}^{d_0|d_1} \! := \A(V, \phi) \, $  if  $ \, d_0\big|d_1 \, \, $  is the superdimension of  $ V \, $,  with  $ \bullet $  denoting the type of  $ \phi \, $.
 \vskip3pt
   Now observe that for any {\sl graded} real structure  $ \phi $  on a finite-dimensional complex superspace  $ V $,  from  $ \, \phi{\big|}_{V_\uno}^2 = -\text{id}_{V_\uno} \, $  it easily follows that  $ V_\uno $  has a  $ \C $--basis  $ \, \big\{ u_1 , \dots , u_\delta , w_1 , \dots , w_\delta \big\} \, $  such that  $ \, \phi(u_i) = +w_i \, $,  $ \, \phi(w_i) = -u_i \, $,  for all  $ \, i = 1 , \dots , \delta \, $.  In particular,  $ \, d_1 := \text{\sl dim}(V_\uno) = 2\,\delta \, $  is even, and  $ \, \phi{\big|}_{V_\uno} \, $  is described   --- as a  $ \C $--linear  map from  $ V_\uno $  to  $ \overline{V}_\uno \, $,  that is  $ V_\uno $  endowed with the conjugate complex structure ---   by the  $ 2 \times 2 $  block matrix
 $ \; \begin{pmatrix}
     \, 0  &  -I_\delta \;  \\
    +I_\delta  &  0 \,
      \end{pmatrix} \; $.
                                                                      \par
   In particular, if  $ V $  is a complex super vector space which is  {\sl entirely odd},  i.e.\  $ \, V = V_\uno \, $,  $ \, V_\zero = 0 \, $,  with graded real structure  $ \phi \, $,  then  $ S(V^*) $  is isomorphic to the complex Grassmann algebra  $ \, \Lambda_\C\big( \xi_1^+ , \dots , \xi_\delta^+ , \xi_1^- , \dots , \xi_\delta^- \big) \, $  in  $ \, 2\,\delta = d_\uno \, $  odd indeterminates   --- where  $ \, \delta := d_1\big/2 = \text{\sl dim}(V_\uno) \big/ 2 \, $  ---   with graded real structure given by  $ \, \phi\big( \xi_i^\pm \big) := \pm \xi_i^\mp \, $  for all  $ \, i = 1 , \dots , \delta \, $.  Note that the  $ A $--points   --- for any  $ \, A \in \salg_\C^{\text{gr}} \, $  ---   of  $ \, \A_{\text{gr},\C}^{0|d_1} := \A(V, \phi) \, $  are given by
  $$  \A_{\text{gr},\C}^{0|d_1}(A)  \,\; = \;\,  \Big\{ {\big( \alpha_i^+ , \alpha_i^- \big)}_{i=1,\dots,\delta} \,\Big|\, \alpha_i^\pm \in A_\uno \, , \, \widetilde{\alpha_i^\pm} = \pm \alpha_i^\mp \, , \; \forall \, i \,\Big\}  $$
or
  $$  \begin{matrix}
   \A_{\text{gr},\C}^{0|d_1}(A)  \,\;  &  = \;\,
\Big\{ \Big( \alpha_i^+ , +\widetilde{\alpha_i^+} \,\Big)_{i=1,\dots,\delta}
\,\Big|\; \alpha_i^+ \in A_\uno \, , \; \forall \; i = 1 , \dots , \delta \,\Big\}  \; =  \cr
   &  = \;\,  \Big\{ \Big(\!\! -\!\widetilde{\alpha_i^-} ,
\alpha_i^- \Big)_{i=1,\dots,\delta} \Big|\; \alpha_i^- \in A_\uno \, , \; \forall \; i = 1 , \dots , \delta \,\Big\}  \phantom{=}
      \end{matrix}  $$

\medskip

%
%
   When a real structure  $ \varphi $  on  $ \cL_V $  is  {\sl standard},  we have the following characterization of  $ \mathcal{L}_V^{\,\varphi} \; $:

\smallskip

\begin{proposition} \label{re-forms-svs}
 Let  $ \varphi $  be a  {\sl standard}  real form on  $ \mathcal{L}_V \, $,  and  $ \phi $  the corresponding real structure on  $ V $.  Then
 \vskip3pt
   \hskip-9pt   {\it (a)}\;\;   $ \mathcal{L}_V^{\,\varphi}(A)  \, := \,  {(A \otimes V)}_\zero^{\varphi_A} \, = \, {\big( A^\re \otimes V^\phi \,\big)}_\zero \, $,  \; with  $ \; A^\re = \big\{\, a \! \in \! A \,\big|\, a = \tilde{a} \,\big\}   \hfill  \forall \; A \in \salg_\C^\st \; $;   \qquad
 \vskip3pt
   \hskip-9pt   {\it (b)}\;\;   $ \mathcal{L}_V^{\,\varphi}(R \otimes \C)  \, = \,  \Hom_{\salg^\st_\C}\big( \C[V] , R \otimes \C \,\big) \, = \, \Hom_{\salg_\R}\big( \R\big[V^\phi\big] , R \,\big)   \hfill  \forall \; R \in \salg_\R \; $.   \qquad
\end{proposition}

\noindent
 {\it Proof}.  {\it (a)}\,  Definitions give  $ \; {(A \otimes V)}_\zero^{\varphi_A} = \big\{\, a \otimes v + \varphi_A (a\otimes v) \,\big|\, |a| = |v| \,\big\} \; $.  Let  $ \, a = a_1 + i\,a_2 \, $,  $ \, v = v_1 + i\,v_2 \, $,  so that  $ \, \widetilde{a} = a_1 - i\,a_2 \, $  and  $ \, \phi(v) = v_1 - i\,v_2 \, $.  Then
  $$  (a_1+ia_2) \otimes (v_1+iv_2) + (a_1-ia_2) \otimes (v_1-iv_2)  \; = \;  2a_1 \otimes v_1-2a_2 \otimes v_2 \in {\big( A^\re \otimes V^\phi \,\big)}_\zero  $$
   \indent   {\it (b)}\,  By  Proposition \ref{gr-rep1}  we have
  $$  \displaylines{
   \mathcal{L}_V^{\,\varphi}(R \otimes \C)  \; = \;  {\big( (R \otimes_\R \C) \otimes_\C V \,\big)}_\zero^{\varphi_{R \otimes \C}}  \; = \;  \Hom_{\smod^\st_\C} \big( V^* , R \otimes_\R \C \big)  \; =   \hfill  \cr
   \hfill   = \;  \Hom_{\salg^\st_\C} \big( S(V^*) \, , R \otimes_\R \C \,\big)  \; = \;  \Hom_{\salg^\st_\C} \big( \C[V] \, , R \otimes_\R \C \,\big)  }  $$
  \indent   On the other hand, by claim  {\it (a)\/}  we have
  $$  \displaylines{
   \mathcal{L}_V^{\,\varphi}(R \otimes \C)  \; = \;  {\big( (R \otimes_\R \C) \otimes_\C V \,\big)}_\zero^{\varphi_{R \otimes \C}}  \; = \;  {\big( {(R \otimes_\R \C)}^{\text{re}} \otimes_\R V^\phi \,\big)}_\zero  \; =   \hfill  \cr
   \hfill   = \;  {\big( R \otimes_\R V^\phi \,\big)}_\zero  \; = \;  \Hom_{\smod_\R}\big( {(V^\phi)}^* , R \,\big)  \; = \;  \Hom_{\salg_\R}\big(\, \R\big[V^\phi\big] \, , R \,\big)
  \quad  \square  }  $$

\medskip

   Notice that, by claim  {\it (b)\/}  of the previous proposition, we have that
%
%
  $ \; \mathcal{L}_V^{\,\varphi}(A) \, = \, h_{V^\phi}(A^\re) \; $
 because in the standard case we have  $ \, A = A^\re \otimes \C \; $;  therefore we can identify  $ \mathcal{L}_V^{\,\varphi} $
%
%
 with the functor  $ h_{V^\phi} $  representing the real super vector space  $ V^\phi \, $.

\smallskip

\begin{observations}  \label{subspaces}
 {\it (a)}\,  If  $ \, V' \subseteq V \, $  is a super vector subspace with real structure induced by  $ (V,\phi) \, $,  then  $ \, \cL_{V'}^\varphi(A) \subseteq \cL_V^\varphi(A) \, $  for all  $ \, A \in {\salg}_\C^\bullet \; $.
                                                       \par
   {\it (b)}\,  If  $ \, V' \subseteq V \, $  is a super vector subspace with  {\sl standard\/}  real structure induced by  $ (V,\phi) \, $,  then  $ \, (V')^\phi \subseteq V^\phi \, $  as real super vector spaces.
\end{observations}

\medskip

\section{Real structures and real forms of supergroups}  \label{sec: Re-struct_Re-forms_x_SGroups}  {\ }
 \vskip-3pt
 We now want to define the notion of real structure  and real form of a supergroup, from different points of view. Let  $ \sgrps_\C \, $ denote the category of complex supergroups.

\subsection{Real structures on supergroups}  \label{subsec: real-structs-sgrps}

   We shall give our definition of real structure using both the sHCp's and the functor of points approach.  We first record a couple of auxiliary observations.

\smallskip

\begin{observation}  \label{obs:cplx-sgrps=real-sgrps}
 Let  $ \bG $  be a complex supergroup, and  $ \Lie(\bG) $  its Lie superalgebra.
   Let  $ \; \mathcal{F} : \salg_\C^\bullet \lra \salg_\C \; $  be the obvious forgetful functor, and  $ \, \mathcal{R} : \smod_\C \!\longrightarrow\! \smod_\R \, $  be the obvious functor of scalar restriction.
 Thinking of  $ \, \bG \, $  as a functor defined on  $ \salg_\C \, $,  we use notation  $ \, \bG^\bullet  := \bG \circ \mathcal{F} \, $.  Then in particular we have   --- with notation of Definition \ref{gr-str2}  ---   $ \; \cL_{\Lie(\bG)} := \mathcal{R} \circ h_{\Lie(\bG)} \circ \mathcal{F} = \mathcal{R} \circ \cL_{\Lie(\bG^\bullet)} \, $.
                                                          \par
   Similarly, for the complex conjugate supergroup  $ \overline{\bG} $  we have a parallel functor  $ \overline{\bG^\bullet} \, $.
%
%
\end{observation}

\smallskip

\begin{lemma}  \label{lem: eq_restr-Lie(G)}
 Let\/  $ \bG $  be a complex supergroup and\/  $ \Lie(\bG) $  its Lie superalgebra, and consider any natural transformation  $ \, \Phi : \bG^\bullet \relbar\joinrel\longrightarrow \overline{\bG^\bullet} \, $   --- which loosely speaking can be equivalently seen as  $ \, \Phi : \overline{\bG^\bullet} \relbar\joinrel\longrightarrow \bG^\bullet \, $.  Then the following are equivalent (notations as in  Observation \ref{obs:cplx-sgrps=real-sgrps}  above):
 \vskip3pt
%
%
   {\it (a)}\,  $ \; \mathcal{R}\big(\Lie(\Phi)\big) : \cL_{\Lie(\bG)} \relbar\joinrel\longrightarrow \cL_{\Lie(\bG)} \; $  is a real structure for\/  $ \cL_{\Lie(\bG)} \, $.
 \vskip3pt
   {\it (b)}\,  $ \; \big( \Phi_{\!A[\epsilon]} \circ \bG^\bullet(v_a) \big)(z) \, = \, \big(\bG^\bullet(v_{\,\widetilde{a}}) \circ \Phi_{\!A[\epsilon]} \big)(z) \; $  for  $ \, A \in \salg_\C^\bullet \; $,  $ \, a \in A_0 \, $,  $ \, z \in {\Lie(\bG)}^\bullet(A) \, $,  with  $ \, v_a : A[\epsilon] \relbar\joinrel\lra A[\epsilon] \, $  given by  $ \, v_a(x+\epsilon\,y) := x + \epsilon\,a\,y \, $.
\end{lemma}

\begin{proof}
 By definition of $ \bG^\bullet(v_a) $   --- see  \cite{ccf}, \S 11.3 ---   we have  $ \,\; \bG(v_a)(z) = a\,.\,z \; $  for all  $ \, z \in \Lie\big(\bG^\bullet\big)(A)  \, $,  that is  $ \bG^\bullet(v_a) $  gives the action of  $ \, a \in \cA_0 \, $  onto  $ \, \Lie\big(\bG^\bullet\big)(A) \, $.  Moreover, by  \cite{ccf}, Ch.\ 11,  we have  $ \; \Phi_{\!A[\epsilon]}(z) = {\Lie(\Phi)}_A(z) \; $  for all  $ \, z \in \Lie\big(\bG^\bullet\big)(A)  \, $.  But then the condition in  {\it (b)\/}  reads
 $ \; {\Lie(\Phi)}_A(a\,.\,z) \, = \, \widetilde{a}\,.\,{\Lie(\Phi)}_A(z) \; $
which (applying  $ \mathcal{R} $)  is equivalent to the conditions in  {\it (a)}.
\end{proof}

\smallskip

   The following definition is inspired by  \cite{pe}:

\smallskip

\begin{definition}  \label{def:re-str_sgroup2}
 {\it (a)}\,  Let  $ \bG $  be a complex supergroup and  $ \Lie(\bG) $  its Lie superalgebra.  We call {\it (generalized) real structure\/}  on  $ \bG $  any natural transformation
 $ \, \Phi : \bG^\bullet \relbar\joinrel\longrightarrow \overline{\bG^\bullet} \, $  such that
 \vskip3pt
   \quad   {\it (a.1)}\,  $ \; \Phi $  is an involution,  i.e.\  $ \;\; \Phi^2 = \text{id}_{{}_{\scriptstyle \bG^\bullet}} \;\; $,
 \vskip3pt
   \quad   {\it (a.2)}\,
%
%
 $ \; \mathcal{R}\big(\Lie(\Phi)\big) : \cL_{\Lie(\bG)} \relbar\joinrel\longrightarrow
\cL_{\Lie(\bG)} \; $  is a real structure for  $ \cL_{\Lie(\bG)} \;\; $.
 \vskip5pt
   {\it (b)}\,  The pairs  $ (\bG,\Phi) $  consisting of a complex supergroup with a real structure on it, along with all morphisms among them that respect the real structures on both sides, form a category that we denote hereafter by  $ \sgrps_\C^\bullet \, $;  moreover, we also denote by  $ \; \mathcal{F} : \sgrps_\C^\bullet \relbar\joinrel\lra \sgrps_\C \, $,  \, slightly abusing the notation, the natural forgetful functor, see  Observation \ref{obs:cplx-sgrps=real-sgrps}{\it (b)}.
\end{definition}

\vskip1pt

%
%


   As complex supergroups correspond to sHCp's (via a category equivalence), we introduce the notion of generalized real structure for the latter.


\begin{definition}  \label{def:sr-str_sHCp}
%
%
 Let  $ (G_+,\fg) $  be a complex sHCp.  We call  {\it (generalized) sHCp real structure\/}  on
$ (G_+,\fg) $  any pair  $ (\Phi_+,\phi) $  such that
 \vskip3pt
   {\it (a)} \;\;  $ \Phi_+ $  is a real structure (in the classical sense) on the complex algebraic group  $ G_+ \, $;
%
%
 \vskip3pt
   {\it (b)} \;\;  $ \phi \, $  is a real structure on the complex Lie superalgebra  $ \fg \, $;
 \vskip3pt
   {\it (c)} \;\;  $ \, {\big( \Lie(\Phi_+) \big)}_{1_{G_+}} \!\! = \,  \phi{\big|}_{\fg_\zero} \;\, $.
 \vskip5pt
  Then, we can define the category of standard or graded real sHCp's  $ \sHCp_\C^\bullet \, $, according to the type of  $ \phi \, $,  whose morphisms are morphisms of sHCp's which preserve the real structures on either side; in addition, once more we have a natural forgetful functor  $ \; \mathcal{F} : \sHCp_\C^\bullet \relbar\joinrel\lra \sHCp_\C \; $, again with a small
abuse of notation   --- see  Observation \ref{obs:cplx-sgrps=real-sgrps}{\it (b)}.
\end{definition}

\vskip1pt

\begin{remark}  \label{rmk: real-struct_(G_+,g)=C-linear/cplx-conj}
 Just like a real structure on a complex vector superspace  $ V $  can be thought of as a special  {\sl  $ \C $--linear map\/}  from  $ V $  to its complex-conjugate    $ \overline{V} $,  or viceversa, similarly a real structure on a complex supergroup  $ \bG $  can be seen as a special morphism from  $ \bG^\bullet $  to its complex-conjugate, denoted by  $ \overline{\bG^\bullet} \, $.  In the same way, a real structure on a complex sHCp   $ (G_+,\fg) $  can be seen as a particular morphism from   $ (G_+,\fg) $  to its complex-conjugate  $ \overline{(G_+,\fg)} $   --- see \cite{cfk1, dm} for more details.
\end{remark}

\medskip

   We show now that the two notions of real structure, that we have introduced, are indeed equivalent,
through the above mentioned correspondence between supergroups and sHCp's.

\medskip

\begin{proposition}  \label{prop:eq_defs-rst_sgrps-sHCp's}
 Let  $ \bG $  be a complex supergroups and  $ (G_+,\fg) $  a complex sHCp that correspond to each other.  Then there is a one-to-one correspondence between real structures on  $ \bG $  and real structures on  $ (G_+,\fg) \, $.  This induces an equivalence of the corresponding categories  $ \sgrps_\C^\bullet $  and  $ \sHCp^\bullet \, $,  which is consistent   --- via the natural forgetful functors ---   with the equivalence between supergroups and sHCp's: in other words, the following diagram of functors (whose horizontal arrows are the above mentioned equivalences) is commutative
  $$  \xymatrix{
   \sgrps_\C \; \ar[rrr]^\simeq
  &  &  &   \; \sHCp_\C  \\
   &  &  &  \\
   \ar[uu]^{\mathcal{F}} \sgrps_\C^\bullet \, \ar[rrr]_\simeq
  &  &  &  \; \sHCp_\C^\bullet \ar[uu]_{\mathcal{F}}  }  $$
\end{proposition}

\begin{proof}
 One way it is clear: if we have a real structure  $ \Phi $  on  $ \bG $  then via the equivalence  $ \; \sgrps_\C \,{\buildrel \cong \over {\relbar\joinrel\relbar\joinrel\lra}}\, \sHCp_\C \; $  we define the pair  $ \; (\Phi_+ , \phi) :=  \big( \Phi{|}_{G_+} , \Lie(\Phi) \big) \; $.
%
%
%
                                                     \par
   In the reverse direction, for a real structure  $ (\Phi_+,\phi) $  on  $ (G_+,\fg) $  we define  $ \; \Phi : \bG^\bullet \relbar\joinrel\lra \overline{\bG^\bullet} \; $  via the reverse equivalence  $ \; \sHCp_\C \,{\buildrel \cong \over {\relbar\joinrel\relbar\joinrel\lra}}\, \sgrps_\C \; $.  Using the explicit form of such an equivalence provided in  \cite{ga1}  or  \cite{ga2},  we only need to define  $ \Phi_A $   --- for each  $ \, A \in \salg_\C^\bullet \, $  ---   on special elements in  $ \, \bG^\bullet(A) := \bG(A) \, $  of the form  $ \, (1 + \xi\,X) \, $,  with  $ \, \xi \in A_\uno \, $, $ \, X \in \fg_\uno \, $;  then the recipe in  \cite{ga1,ga2}  for them prescribes  $ \,\; \Phi( 1 + \xi\,X ) \, := \, 1 + \widetilde{\xi} \, \phi(X) \; $.
\end{proof}


   In the next result we explain real structures for supergroups described as super-ringed spaces.

\medskip

\begin{proposition}  \label{phi-prop}
 Let\/  $ \bG  = \big( |G| \, , \cO_\bG \big) \, $,
be a complex algebraic supergroup,  $ \, G_+ = \big( |G| \, , \cO_\bG/\cJ \big) \, $  its reduced subgroup, and  $ \Phi_+ $  a real structure on  $ G_+ \, $.  Then there exists a bijection between
 \vskip3pt
   (i)\;  real standard, resp.\ graded, structures  $ \Phi $  on  $ \bG $  such that  $ \, \Phi{\big|}_{G_+} = \Phi_+ \; $;
 \vskip1pt
   (ii)\;  antilinear sheaf morphisms  $ \; {\big\{ \cO_\bG\big(\Phi_+^{-1}(U)\big) \!\longrightarrow \cO_\bG(U) \,\big\}} \; $  which are involutions on the even part and whose square is plus the identity, resp.\ minus the identity, on the odd part.
 \vskip3pt
 In particular, when  $ \bG $  is  {\sl affine},  a real structure on  $ \bG $  is equivalently given by an antilinear morphism  $ \, \C[\bG] \lra \C[\bG] \, $,  \,where  $ \C[\bG] $  is the superalgebra of global sections on  $ \bG \, $,  which reduces to  $ \Phi_+^* $  on the reduced algebra  $ \, \C[\bG]\big/J \, $.
\end{proposition}

\begin{proof}
 We give just a sketch of the argument (for more details, see  \cite{dm, cfk1}).
  By  Proposition \ref{prop:eq_defs-rst_sgrps-sHCp's},  $ \bG $  corresponds to the sHCp  $ (G_+,\fg) $   --- where  $ \, \fg = \Lie(\bG) \, $  as usual ---   and any real structure  $ \Phi $  on  $ \bG $  as in  {\it (i)\/} corresponds to a real structure  $ (\Phi_+,\phi) $  on the sHCp  $ (G_+,\fg) \, $.  In this setup, the structure sheaf  $ \cO_\bG $  of  $ \bG $  can be described  (cf.\ \cite{cf})  as
\begin{equation}  \label{eq: struct-sheaf_G}
  \qquad   \cO_\bG(U)  \; = \;  \Hom_{\mathfrak{U}(\fg_0)}\big(\, \mathfrak{U}(\fg) \, , \cO_{G_+}(U) \big)   \qquad  \text{for all open  $ U $  in\ }  G_+
\end{equation}
 Now, starting from a real structure  $ \Phi $  on  $ \bG $  as in  {\it (i)},  hence from a real structure  $ (\Phi_+,\phi) $  on  $ (G_+,\fg) \, $,  note that the antilinear morphism  $ \; \phi : \fg \lra \fg \; $  extends uniquely to an antilinear morphism  $ \; \mathfrak{U}(\phi) : \mathfrak{U}(\fg) \lra \mathfrak{U}(\fg) \, $.  For each open  $ U $  in  $ G_+ \, $,  this gives a map  $ \; f \mapsto \! {\big(\Phi_+^*\big)}_U \circ f \circ \, \mathfrak{U}(\phi) \; $
%
%
 from  $ \; \cO_{\bG}(U) := \Hom_{\mathfrak{U}(\fg_0)}\big( \mathfrak{U}(\fg) \, , \cO_{G_+}(U) \!\big) \; $  to  $ \; \cO_{\bG}\big(\Phi_+^{-1}(U)\!\big) := \Hom_{\mathfrak{U}(\fg_0)}\big( \mathfrak{U}(\fg) \, , \cO_{G_+\!}\big(\Phi_+^{-1}(U)\big) \hskip-2pt\big) \; $   --- where  $ \; {\big\{ {\big(\Phi_+^*\big)}_U : \cO_{G_+}(U) \lra \cO_{G_+\!}\big(\Phi_+^{-1}(U)\big) \big\}}_U \; $  is
 the built-in, antilinear sheaf morphism.
                                                                   \par
   The construction of the inverse map is left to the reader.
\end{proof}

\medskip

\begin{remark}  \label{phi-rem}
 Let  $ \bG $  be an  {\sl affine\/}  complex (algebraic or Lie) supergroup, and let  $ \C[\bG] $  be the corresponding Hopf superalgebra.  Then  Proposition \ref{phi-prop}  guarantees that any (generalized) real structure on  $ \bG \, $,  say  $ \Phi \, $,  bijectively corresponds to a (generalized) real structure on the Hopf superalgebra  $ \C[\bG] $   --- cf.\ Definition \ref{super-str1};  we denote this last structure by  $ \, \varphi^{-1} : \overline{\C[\bG]} \relbar\joinrel\lra \C[\bG] \, $.
                                                                      \par
   As now  $ \bG $  is affine, its functor of points is representable, and we can describe it in detail.  Identifying  $ \bG $  with its functor of points, and the real structure  $ \Phi $  with a natural transformation
 $ \, {\big\{\, \bG^\bullet(A) \,{\buildrel \Phi_A \over {\relbar\joinrel\lra}}\, \overline{\bG^\bullet}(A) \,\big\}}_{A \in \salg^\bullet_\C} \, $,
 the real structure  $ \, \C[\bG] \,{\buildrel \varphi \over {\relbar\joinrel\lra}}\, \overline{\C[\bG]} = \C\big[\,\overline{\bG}\,\big] \, $  corresponding to  $ \Phi $  is given by  $ \, \varphi := {\Phi_{\C[\bG]}\big( \text{\sl id}_{\C[\bG]} \big)}^{-1} \, $.  Conversely, given  $ \, \C[\bG] \,{\buildrel \varphi \over {\relbar\joinrel\lra}}\, \overline{\C[\bG]} = \C\big[\,\overline{\bG}\,\big] \, $,  the correspon\-ding real structure  $ \Phi $  on  $ \bG $  is given (as a natural transformation) by  $ \, \Phi_A := (-) \circ \varphi^{-1} \, $,  i.e.
  $$  \displaylines{
   \Phi_A : \, \bG^\bullet(A) = \bG(A) := \Hom_{\salg_\C}\big( \C[\bG] \, , A \big) \relbar\joinrel\longrightarrow \Hom_{\salg_\C}\big( \C\big[\,\overline{\bG}\,\big] \, , A \big) =: \overline{\bG}(A) = \overline{\bG^\bullet}(A)  \cr
   \qquad \qquad \quad   g_{{}_A}
\,\raise1,35pt\hbox{$ \scriptscriptstyle | $}\hskip-5,6pt\relbar\joinrel\relbar\joinrel\relbar\joinrel\relbar\joinrel\relbar\joinrel\relbar\joinrel\relbar\joinrel\relbar\joinrel\relbar\joinrel\relbar\joinrel\relbar\joinrel\relbar\joinrel\relbar\joinrel\relbar\joinrel\relbar\joinrel\longrightarrow\, \Phi_A(g_{{}_A}) \, := \, g_{{}_A} \!\circ \varphi^{-1}  }  $$
for all  $ \, A \in \salg_\C \, $,  taking into account   --- cf.\ Remark \ref{rmk: real-struct_(G_+,g)=C-linear/cplx-conj}  ---   that any real structure on  $ \bG $  can be seen as a special supergroup morphism from  $ \, \bG^\bullet \, $  to  $ \, \overline{\bG^\bullet} \, $
%
%
 (the complex-conjugate of  $ \, \bG^\bullet \, $).
                                                                      \par
%
%
   Now we modify the natural transformation  $ \; \Phi := {\big\{ \Phi_A \big\}}_{A \in \salg_\C^\bullet} : \bG^\bullet \!\relbar\joinrel\lra \overline{\bG^\bullet} \; $  above, by setting  $ \; \Phi^\bullet_A := \widetilde{(\ )}_A \circ \Phi_A \; $  for all  $ \, A \in \salg_\C^\bullet \, $,  \,that is in detail
\begin{equation}  \label{eq: Z_4-action on G(A)}
  \Phi^\bullet_{\!A}(g_{{}_A}) := \widetilde{(\ )}_A \circ g_{{}_A} \!\circ \varphi^{-1}  \qquad \forall \;\; g_{{}_A} \in \bG^\bullet(A) := \Hom_{\salg_\C}\big( \C[\bG] \, , A \big)
\end{equation}
since  $ \, \bG^\bullet(A) := \bG(A) \, $.  Note that  $ \, \Phi^\bullet_{\!A}(g_{{}_A}) \in \Hom_{\salg_\C}\big( \C[\bG] \, , A \big) =: \bG^\bullet(A) \, $  since each  $ \Phi^\bullet_{\!A}(g_{{}_A}) $  is now  {\sl  $ \C $--linear},  so  $ \, \Phi^\bullet_{\!A} \, $  is a group morphism from  $ \bG^\bullet(A) $  to  $ \bG^\bullet(A) \, $.  All these  $ \Phi^\bullet_{\!A} $'s  define a natural transformation  $ \, \Phi^\bullet := {\big\{ \Phi^\bullet_{\!A} \big\}}_{A \in \salg_\C^\bullet} \, $  from  $ \bG^\bullet $  to itself: in the following, whenever  $ \bG $  is affine by  {\sl real structure on\/} $ \bG^\bullet $  we shall mean exactly this supergroup endomorphism  $ \, \Phi^\bullet : \bG^\bullet \relbar\joinrel\lra \bG^\bullet \, $.
\end{remark}

\medskip

\subsection{Real forms of supergroups}

   We now turn to the definition of (generalized) real forms for supergroups.

\smallskip

\begin{definition}  \label{def:re-form_sgrps_1}
   Let  $ \, (\bG,\Phi) $  be a complex supergroup with real structure, and  $ \, \bG^\bullet := \bG \circ \cF \, $  as above.
 We define  {\it (generalized) real form functor\/}  (``standard/graded'', according to  $ \Phi $)  of  $ \, (\bG,\Phi) \, $,  or  {\sl ``real form functor of\/  $ \bG $  with respect to\/  $ \Phi $''},  the subgroup functor  $ \bG^\Phi $  of  $ \bG^\bullet $
%
%
 defined by
 \vskip-5pt
  $$  \bG^\Phi : \salg_\C^\bullet \!\longrightarrow \grps \; ,  \quad  A \, \mapsto \, \bG^\Phi(A) := {\bG^\bullet(A)}^{\Phi^\bullet_{\!A}} \; ,  \;\;\; \bG^\Phi(f) := f{\big|}_{{\bG^\bullet(A)}^{\Phi^\bullet_{\!A}}}  $$
--- for every  $ \, A \, , B \in \salg_\C^\bullet \, $,  $ \, f \in \Hom_{\salg_\C^\bullet}(A\,,B\,) \, $  ---   where we denote by
  $$  {\bG^\bullet(A)}^{\Phi^\bullet_{\!A}}  \,\; := \;\, \big\{\, g \! \in \! \bG^\bullet(A) \,\big|\, \Phi^\bullet_{\!A}(g) = g \,\big\}  $$
 the subgroup of  $ \Phi^\bullet_{\!A} $--invariants.  Hereafter we are tacitly  {\sl identifying}   --- as it is always possible, by general theory ---   the abstract groups  $ \, \bG^\bullet(A) := \bG(A) \, $  and  $ \, \overline{\bG^\bullet}(A) := \overline{\bG}(A) \, $.
\end{definition}

\smallskip

\begin{proposition}  \label{prop: repr-G^Phi}
 Let\/  $ \bG $  be an affine complex supergroup with (generalized) real structure  $ \Phi \, $.  Then the functor\/  $ \bG^\Phi $  is representable.
\end{proposition}

\begin{proof}
 As  $ \bG $  is affine, let  $ \, \C[\bG] \in \salg_\C \, $  be the Hopf superalgebra representing it, as  a functor from  $ \salg_\C $  to  $ \grps \, $:  then by  Proposition \ref{phi-prop},  there exists a real structure  $ \, \varphi : \C[\bG]  \relbar\joinrel\lra \C[\bG] \, $   --- which corresponds uniquely to  $ \Phi $  ---   so that  $ \, \big(\C[\bG] \, , \varphi \big) \in \salg_\C^\bullet \, $.  Now  Definition \ref{def:re-form_sgrps_1}  together with Remark \ref{phi-rem}  yield   --- for every  $ \, A \in \salg_\C^\bullet \, $  ---
  $$  \displaylines{
   {\bG(A)}^{\Phi_{\!A}^\bullet}  := \,  \big\{\, g \in \bG^\bullet(A) \,\big|\, \Phi_{\!A}^\bullet(g) \! = g \,\big\}_{{}_{\phantom{|}}}  \! = \,  \big\{\, g \in \bG^\bullet(A) \,\big|\, \widetilde{(\ )}_A \circ g \circ \varphi^{-1} \! = g \,\big\}_{{}_{\phantom{|}}}  \! =   \hfill  \cr
   \hfill   = \;  \big\{\, g \in \bG^\bullet(A) \,\big|\, \widetilde{g(u)} = g\big(\varphi(u)\big) \, ,  \; \forall \; u \in \C[\bG] \big\}  \; = \;  \Hom_{\salg_\C^\bullet}\big( \C[\bG] \, , A \,\big)  }  $$
\noindent
 because the condition  $ \, \widetilde{g(u)} = g\big(\varphi(u)\big) \, $   --- for  $ \, u \in \C[\bG] \, $ ---   means that the superalgebra morphism  $ \, g : \C[\bG] \relbar\joinrel\lra A \, $  preserves the real structure on both sides, hence  $ \, g \in \Hom_{\salg_\C^\bullet}\big( \C[\bG] \, , A \big) \, $.
\end{proof}

\smallskip

\begin{observation}
%
%
%
 Let us consider a  {\sl standard\/}  real structure  $ \Phi $ on a complex  {\sl affine\/}  supergroup  $ \bG \, $,  i.e.\  $ \, (\bG, \Phi) \in \sgrps_\C^\st \, $,  and let  $ \varphi $  be the corresponding real structure on  $ \C[\bG] \, $.  As  $ \Phi $  is standard, the same is true for  $ \varphi $  as well: then each  $ \, f \in \C[\bG] \, $  has a unique splitting as  $ \, f = f_+ + i \, f_- \, $  with  $ \, \varphi(f_\pm) = f_\pm \, $.  Using this, the relation  \eqref{eq: Z_4-action on G(A)}  and   --- for all  $ \, A \in \salg_\C^\st \, $  ---   the identity  $ \; \Hom_{\salg_\C^\st}\big(\, \C[\bG] \, , A \,\big)  \, = \,  {\Hom_{\salg_\C}\big(\, \C[\bG] \, , A \,\big)}^{\Phi_{\!A}^\bullet} \; $  and  Proposition \ref{prop: repr-G^Phi},  one finds that
  $$  \bG^\Phi\big( \C \otimes_\R R \,\big)  \,\; := \;\,  \Hom_{\salg_\C^\st}\big(\, \C[\bG] \, , \C \otimes_\R R \,\big)  \,\; = \;\,  \Hom_{\salg_\R}\big(\, {\C[\bG]}^\varphi , R \,\big)  $$
for all  $ \, R \in \salg_\R \, $.  This gives us a description of the  {\sl real\/}  supergroup functor  $ \; R \mapsto \bG^\Phi\big( \C \otimes_\R R \,\big) \; $   --- for all  $ \, R \in \salg_\R \, $  ---   which is the real form  $ \bG^\Phi $  (of  $ \bG \, $)  when seen as a  {\sl real\/}
supergroup.
                                                             \par
   This is the analog, in some sense, of Proposition \ref{re-forms-svs}  for super vector spaces.
\end{observation}

\smallskip

\subsection{Functor of points of real forms}  \label{sec: functor-pts-real-forms}

 In this section we describe in detail the real form of a supergroup, using the functor of points approach.  To begin with, we shortly recall the following.
                                              \par
   For the standard functor
 $ \; \cK : \sgrps_\C \!\relbar\joinrel\relbar\joinrel\lra \sHCp_\C \; $  we choose a specific quasi-inverse functor  $ \, \cH : \sHCp_\C \!\relbar\joinrel\relbar\joinrel\lra \sgrps_\C \, $,  namely the second one described in  \cite{ga2},  therein denoted by  $ \Psi^e $.  Via the latter, for every  $ \, \bG \in \sgrps_\C \, $  and  $ \, A \in \salg_\C \, $  the group  $ \bG(A) $  is described as
\begin{equation}  \label{eq: G-descr_via-sHCp}
  \bG(A)  \; = \;  G_+(A_\zero) \cdot \exp\!\big( A_\uno \!\otimes_\C \fg_\uno \big)  \; \cong \;  \Big( G_+ \times \mathbb{A}_\C^{0|d_\uno}\Big)(A)
\end{equation}
 where  $ \, \exp\!\big( A_\uno \!\otimes_\C \fg_\uno \big) \! := \! \big\{ \exp(\cY) \,\big|\, \cY \! \in \! A_\uno \!\otimes_\C \fg_\uno \big\} \, $,  $ \, d_\uno := \text{\sl dim}(\fg_\uno) \, $,  and the symbol  ``$ \, \cong \, $''  on the right just means that  $ \bG $  and  $ \, G_+ \times \mathbb{A}_\C^{0\,|d_\uno} \, $  are isomorphic as supermanifolds and as groups.  In particular, formula  \eqref{eq: G-descr_via-sHCp}  means that each  $ \, g \in \bG(A) \, $  has a unique expression of the form
\begin{equation}  \label{eq: g(elem)-descr_via-sHCp}
  g  \; = \;  g_+ \cdot \exp(\cY)
\end{equation}
for some unique  $ \, g_+ \in G_+(A) \, $  and  $ \, \cY \in A_\uno \!\otimes_\C \fg_\uno \, $.  Now, let  $ \Phi $  be the chosen real structure on  $ \bG \, $,  and  $ \, \big( \Phi_+ \, , \phi \big) \, $  its corresponding real structure on  $ \, \big( G_+ \, , \fg \big) \, $;  then the action of  $ \Phi $  on  $ \, g \in \bG(A) \, $  reads   --- through  $ \eqref{eq: g(elem)-descr_via-sHCp} $,  and setting  $ \, \varphi_A := \widetilde{( - )} \otimes \phi \, $  (cf.\  Theorem \ref{eq-real})  ---   as follows:
\begin{equation}  \label{eq: Phi(g)-descr_via-sHCp}
  \Phi(g)  \; = \;  \Phi_+(g_+) \cdot \exp\!\big( \varphi_A(\cY) \big)
\end{equation}

\vskip5pt

   We are now ready for the main result in this section, which is Theorem A in Sec.\ \ref{intro}.

\vskip15pt

\begin{theorem}  \label{real-form_G}
 Let  $ \, \big(\bG,\Phi\big) \in \sgrps_\C^\bullet \, $.  Then the real form  $ \, \bG^\Phi $  of  $ \, \bG $  is explicitly described as
\begin{equation}  \label{eq: descr-G^Phi}
  \qquad   \bG^\Phi(A)  \; = \;  G_+^{\Phi_+}(A_\zero) \cdot \exp\!\Big( {( A_\uno \!\otimes_\C \fg_\uno)}^{\widetilde{(-)} \otimes \phi} \Big)   \qquad \qquad  \forall \;\; A \in \salg_\C^\bullet
\end{equation}
   Moreover, the factorization is  {\sl direct}:
each $ \; g \in \bG(A)$  has a  {\sl unique}  factorization of the form
  $$  g = g_+ \!\cdot \exp(\cY) \quad \hbox{with}  \quad \, g_+ \in G_+^{\Phi_+}(A_\zero) \quad \hbox{and}  \quad \cY \in {( A_\uno \!\otimes_\C \fg_\uno)}^{\widetilde{(-)} \otimes \phi}  $$
%
 In particular we have  $ \; \bG^\Phi \cong \, G_+^{\Phi_+} \! \times \mathbb{A}_{\bullet,\C}^{0\,|d_1} \, $  (cf.\ \S\ \ref{sec_affine-real-spaces}),  hence the functor  $ \, \bG^\Phi $  is  {\sl representable}.
%
\end{theorem}

\begin{proof}
   Given  $ \, g \in \bG(A) \, $,  with factorization  $ \, g = g_+ \!\cdot \exp(\cY) \, $  as in  \eqref{eq: g(elem)-descr_via-sHCp},  by  \eqref{eq: Phi(g)-descr_via-sHCp}  we have
  $$  g \in \bG^\Phi(A)  \;\; \iff \;\;  \Phi(g) = g  \;\; \iff \;\;  \Phi_+(g_+) \cdot \exp\!\big(\varphi_A(\cY)\big) \; = \; g_+ \!\cdot \exp(\cY)  $$
and the rightmost condition is equivalent to  $ \; \Phi_+(g_+) = g_+ \, $  {\sl together with}  $ \, \exp\!\big(\varphi_A(\cY)\big) = \exp(\cY) \; $,  i.e.\  $ \; \Phi_+(g_+) = g_+ \, $  {\sl and}  $ \, \varphi_A(\cY) = \cY \; $,  which means $ \; g_+ \in G_+^{\Phi_+}(A) \, $  {\sl and}  $ \, \cY \in {( A_\uno \!\otimes_\C \fg_\uno)}^{\varphi_A} \; $.  Then
  $$  g = g_+ \!\cdot \exp(\cY) \, \in \, \bG^\Phi(A)  \quad \iff \quad
   \begin{cases}
      \;\; g_+ \in G_+^{\Phi_+}(A)  \\
      \;\; \cY \in {( A_\uno \!\otimes_\C \fg_\uno)}^{\varphi_A}
   \end{cases}  $$
as claimed.  Moreover, the factorization  $ \, g = g_+ \!\cdot \exp(\cY) \, $  is unique by construction.
%
%
\end{proof}

\alert{
  We end this section with a remark regarding the more general setting
  of supermanifolds, that we shall not pursue directly in this paper.

  \begin{remark}
    The sheaf theoretic characterization of standard and graded real
    forms of a supergroup as in  Proposition \ref{phi-prop}  can be extended, almost
    immediately, to give a well posed more general definition of real forms
    (standard and graded) of supermanifolds.
    \end{remark}}

\section{Hermitian forms and unitary Lie superalgebras}

 We introduce now a suitable notion of  {\sl unitary Lie superalgebra},  which is a special real form of  $ \fgl(V) $  associated with a Hermitian form
on the superspace  $ V \, $.

\smallskip

\subsection{Super Hermitian Forms}

%
%
 \label{super-Herm-form}
%
%
 We begin with the definition of  \textit{super Hermitian form\/}  on a complex super vector space  $ V \, $:  this is a map  $ \, B : V \times V \relbar\joinrel\lra \C \, $  which is  $ \C $--linear in the first entry,  $ \C $--antilinear  in the second entry, and such that
  $$  B(x,y)  \; =\;  {(-1)}^{|x||y|} \, \overline{B(y,x)}   \qquad \qquad  \forall \;\; x \, , y \in \big( V_\zero \cup V_\uno \big)  $$
 In addition, we say that  $ B $  is  \textit{consistent\/}  if  $ \, B(x,y) = 0 \, $  for any homogeneous $ x $  and  $ y $  of different parity  (see \cite{vsv},  pg.\ 112, for more details).  From now on we assume  $ B $  to be consistent.

\medskip

   We can write any consistent super Hermitian form  $ B $  as  $ \, B = B_\zero + i \, B_\uno \, $,  where each  $ \, B_{\overline{z}} := {(-i)}^{\overline{z}} \, B{\Big|}_{V_{\overline{z}} \times V_{\overline{z}}} \, $  is an Hermitian form (in the classical, non-super sense) on the vector space  $ V_{\overline{z}} \, $,  for  $ \, \overline{z} \in \Z_2 \, $.  Notice then that  $ \, B' = B_\zero - i \, B_\uno \, $  is also another super Hermitian form on  $ V \, $.
 \smallskip
   We say that  $ B $  is  \textit{non degenerate\/}  if both the  $ B_{\overline{z}} $'s  are non degenerate; similarly,  $ B $  is  \textit{positive definite\/}  if both the  $ B_{\overline{z}} $'s  are positive definite: in this case we write  $ \, B > 0 \, $.  If instead  $ \, B_\zero > 0 \, $  and  $ \, B_\uno < 0 \, $,  then  $ B' $  (as defined above) is a positive definite super Hermitian form, instead of  $ B \, $.

\smallskip

\begin{example}\label{unit-ex0}
 Let  $ \, V := \C^{m|n} \, $.  We can define on  $ V $  two super Hermitian forms, say  $ B^+_V $  and  $ B^-_V \, $,  given by
\begin{equation}  \label{sHer-B_V}
  B^\pm_V\big( (z,\zeta) \, , \big(z',\zeta'\,\big)\big)  \; := \;\,  z \cdot \overline{z'} \, \pm \, i \, \zeta \!\cdot \overline{\zeta'}
\end{equation}
 where  $ \, z , z' \in \C^m \, $,  $ \, \zeta , \zeta' \in \C^n \, $,  while  $ \, z \cdot z' \, $  and  $ \, \zeta \!\cdot \zeta' \, $  are the usual
scalar products in  $ \C^m $  and  $ \C^n $.
\end{example}

\smallskip

   We recall also the notion of  {\sl supersymmetric (bilinear) form\/}  on a complex super vector space  $ V \, $:  it is a  $ \C $--bilinear  map  $ \; \langle\,\ ,\ \rangle : V \times V \lra \C \; $  such that
  $$  \langle x \, , y \rangle  \; =\;  {(-1)}^{|x||y|} \, \langle y \, , x \rangle   \qquad
\qquad  \forall \;\; x \, , y \in \big( V_\zero \cup V_\uno \big)  $$
 Again, we say that the form  $ \, \langle\,\ ,\ \rangle \, $  is  \textit{consistent\/}  if  $ \, \langle x \, , y \rangle = 0 \, $  for any homogeneous $ x $  and  $ y $  of different parity.  From now on we assume any such form  $ \, \langle\,\ ,\ \rangle \, $  to be consistent.

\medskip

  Now let  $ \phi $  be a real structure on  $ V $  and  $ \, \langle \,\ , \ \rangle \, $  be any  $ \C $--bilinear  form on  $ V \, $.  We say that  {\it the form  $ \, \langle \,\ , \ \rangle \, $  is  $ \phi $--invariant}   --- or just  {\it invariant}  ---   if it is a morphism of superspaces with real structures  (i.e.\ of  $ \Z_4 $--modules),  that is
 $ \,\; \overline{\langle\, v \,, w \,\rangle}  \, = \,  \big\langle \phi(v) \,, \phi(w) \big\rangle \;\, $  for all  $ \; v, w \in V \; $.
 Then we have the following link with Hermitian forms on  $ V \, $,  which follows by direct computation:

\medskip

\begin{proposition}  \label{herm-bil}
 Let  $ \, (V,\phi) \in \smod_\C^\bullet \, $  and let  $ \, \langle \,\ , \ \rangle \, $  be a  $ \phi $--invariant,  consistent, supersymmetric,  $ \C $--bilinear  form on  $ V \, $.
%
%
 Then
\beq  \label{herm-susy}
   \, B^\pm_\phi(x,y) := {(\pm i\,)}^{\nu_\phi \, |x|\,|y|} \, \big\langle x \, , \phi(y) \big\rangle  \qquad \text{with} \qquad
   \nu_\phi \, := \,
     \begin{cases}
   \; \zero  &  \! \text{if  $ \phi $ is standard} \\
   \; \uno  &  \! \text{if  $ \phi $ is graded}
     \end{cases}
\eeq
 defines two consistent super Hermitian forms  $ B^+_\phi $  and  $ B^-_\phi $  on  $ V $ (which coincide if\/  $ \phi $  is standard).
\end{proposition}

%
%

\smallskip

\begin{observation}
 When  $ \phi $  is {\sl graded},  we can write the super Hermitian form  $ B^\pm_\phi $  in  \eqref{herm-susy}  as
  $$  B^\pm_\phi(x,y)  \; = \;  \big\langle x_0 \,, \phi(y_0) \big\rangle \, \pm \, i \, \big\langle x_1 \,, \phi(y_1) \big\rangle  \; = \;  B_0(x_0,y_0) \, \pm \, i \, B_1(x_1\,,y_1)  $$
where  $ \, B_0(x_0,y_0) := \big\langle x_0 \,, \phi(y_0) \big\rangle \, $  and  $ \, B_1(x_1,y_1) := \big\langle x_1 \,, \phi(y_1) \big\rangle \, $  are ordinary Hermitian forms (in the classical, non-super sense) on  $ V_0 $  and  $ V_1 $  both considered as plain complex vector spaces   --- i.e.\ forgetting their super structure.  Similarly, if  $ \phi $  is  {\sl standard\/}  we can write  $ \, B_\phi := B^\pm_\phi \, $  as
  $$  B_\phi(x,y)  \; = \;  \big\langle x_0 \,, \phi(y_0) \big\rangle \, + \, i \, \big(\! -i \,
\big\langle x_1 \,, \phi(y_1) \big\rangle \big)  \; = \;  B_0(x_0,y_0) \, + \, i \, B_1(x_1\,,y_1)  $$
 where  $ \; B_0(x_0,y_0) := \big\langle x_0 \,, \phi(y_0) \big\rangle \; $  and  $ \; B_1(x_1,y_1) := -i \, \big\langle x_1 \,, \phi(y_1) \big\rangle \; $  are both ordinary Hermitian forms on  $ V_0 $  and  $ V_1 $  respectively (now seen as plain complex vector spaces).
\end{observation}

\smallskip

   We end this section with some examples of real structures in  $ \C^{m|n} $,  to be used later on.

\smallskip

\begin{examples}  \label{unit-ex}
 \vskip1pt
 Let  $ \, V := \C^{m|n} = \C^{m|2t} \, $  with  $ \, n = 2\,t \in 2\,\N_+ \, $;
\,we consider on it the {\sl standard\/}  and {\sl graded\/}  real structures  $ \phi_\st $  and  $ \phi_\gr $  defined by

\beq  \label{can_st-real-str}
  \phi_\st : \C^{m|n} \relbar\joinrel\relbar\joinrel\lra \C^{m|n} \;\; ,
\qquad  (z \, , \zeta) \, \mapsto \, \phi_\st(z \, , \zeta) :=
\big(\, \zbar \, , \zetabar \,\big)
\eeq
\beq  \label{can_gr-real-str}
  \begin{aligned}
     &  \phi_\gr \, : \, \C^{m|2t} \relbar\joinrel\relbar\joinrel\relbar\joinrel\relbar\joinrel\relbar\joinrel\relbar\joinrel\relbar\joinrel\relbar\joinrel\relbar\joinrel\relbar\joinrel\relbar\joinrel\lra \, \C^{m|2t}  \cr
     &  \;\;  (z\,,\zeta_+\,,\zeta_-)  \; \mapsto \;  \phi_\gr(z\,,\zeta_+\,,\zeta_-)
:= \big(\, \zbar \, , +\overline{\zeta_-} \, , -\overline{\zeta_+} \,\big)
  \end{aligned}
\eeq

 \vskip3pt
   Now we fix in  $ \C^{m|2t} $  the bilinear  form
 $ \; {\langle\,\ ,\ \rangle}_V : \C^{m|2t} \times \C^{m|2t} \!\relbar\joinrel\lra \C \; $
 defined by
  $$  {\big\langle (z\,,\zeta_+\,,\zeta_-) \, , \big( z' \, , \zeta'_+ \, , \zeta'_- \big) \big\rangle}_V  \,\; := \;\,  z \cdot z' \, + \, \zeta_+ \!\cdot \zeta'_- - \, \zeta_- \!\cdot \zeta'_+  $$
 (notation as before).  A moment's check shows that the form  $ \, {\langle\,\ ,\ \rangle}_V \, $  fulfills the following:
 \vskip1pt
   ---  \; {\it (1)} \quad  it is supersymmetric,
 \vskip1pt
   ---  \; {\it (2)} \quad  $ \overline{\langle x \, , y \rangle_V} \, = \, {\big\langle \phi(x) \, , \phi(y) \big\rangle}_V \;\, $  for all  $ \, x , y \in V $,  \,for both  $ \, \phi \in \big\{ \phi_\st \, , \phi_\gr \big\} \; $.

\vskip5pt

   According to  Proposition \ref{herm-bil},  there exist two pairs of super Hermitian forms on  $ \, V := \C^{m|2t} \, $  associated with the form  $ \, {\langle\,\ ,\ \rangle}_V \, $  and the real structures  $ \phi_\st $  and  $ \phi_\gr \, $,  namely
\begin{itemize}
  \item  {\sl Standard case\/  {\rm (the sign being irrelevant)}:}
\begin{equation}  \label{Bphist-canon}
   B^\pm_{\phi_\st}\!\big( (z,\zeta_+,\zeta_-) \, , \big(z',\zeta'_+,\zeta'_-\big) \big)
\,\; = \;\,  z \cdot \overline{z'} \, + \, \zeta_+ \cdot \overline{\zeta'_-} \, - \, \zeta_- \cdot \overline{\zeta'_+}   \quad
\end{equation}
\item  {\sl Graded case:}
\begin{equation}  \label{Bphigr-canon}
  B^\pm_{\phi_\gr}\!\big( (z,\zeta_+,\zeta_-) \, , \big(z',\zeta'_+,\zeta'_-\big) \big)
\,\; = \;\,  z \cdot \overline{z'} \, \mp \, i \, \big(\, \zeta_+ \!\cdot \overline{\zeta'_+} \, + \, \zeta_- \!\cdot \overline{\zeta'_-} \,\big)
\end{equation}
\end{itemize}

\smallskip

   {\sl Note\/}  that, using the compact notation  $ \, \zeta := (\zeta_+ , \zeta_-) \, $,  we can re-write the forms  $ B^\pm_{\phi_\gr} $  as
  $$  B^\pm_{\phi_\gr}\!\big( (z\,,\zeta) \, , \big( z' \, , \zeta' \,\big) \big)  \,\; = \;\,  z \cdot \overline{z'} \, \mp \, i \, \zeta \!\cdot \overline{\zeta'}  $$
which looks like  \eqref{sHer-B_V}  in the standard case, up to switching signs.
\end{examples}

\smallskip

\begin{remark}
   It is worth stressing that not all Hermitian forms can be realized as  $ B_\phi $  as in  Proposition \ref{herm-bil};  in fact, for any such  $ B_\phi $  the odd part of the superspace  $ V $  must be even dimensional.  So, for example, the forms in  \eqref{sHer-B_V}  on  $ \C^{m|n} $  for  {\sl odd\/}  $ n $  cannot be realized as a  $ B_\phi \, $.
                                                    \par
   Nevertheless, we will have a particular interest for Hermitian forms on  $ \fgl(m|n) \, $:  note that for this superspace the odd part has dimension  $ 2\,m\,n \, $.
\end{remark}

\medskip

\subsection{Functorial Hermitian forms}

 We introduce now the functorial counterpart of the notion of super Hermitian form.

\medskip

\begin{definition}  \label{herm-def}
%
%
%
 Given  $ \, (V,\phi) \in \smod^\bullet_\C \, $   --- for any  $ \, \bullet \in \{\text{\rm st} , \text{\rm gr}\} \, $  ---   recall that the functor  $ \, \cL_V : \salg_\C^\bullet \lra \smod_\R \, $  has values into the category  $ \smod_\R $  of real super vector spaces with  $ \Z_2 $--grading  given by  $ \, {\big( \cL_V(A) \big)}_{\overline{z}} := A_{\overline{z}} \otimes_\C V_{\overline{z}} \, $   --- for each  $ \, {\overline{z}} \in \Z_2 \, $;  {\sl for this grading, we denote with  $ \, [v] := \overline{z} \, $  the degree of a homogeneous vector  $ \, v \in {\big( \cL_V(A) \big)}_{\overline{z}} \; $.}
                                                                \par
   We call  {\it functorial Hermitian form\/}  (or just  {\it Hermitian form\/})  $ \cB $  on  $ \, \cL_V : \salg_\C^\bullet \longrightarrow \smod_\R \, $  any natural transformation  $ \, \cB : \cL_V \times \cL_V \lra \cL_\C \, $  such that
\begin{enumerate}
  \item  $ \, \cB \, $  is  $ \cA_\zero $--linear  on the left, i.e.\ it is left-additive and such that   \hfill\break
      \phantom{A} \qquad   $ \; \cB(a\,X,Y)\, = \, a \, \cB(X,Y) \; $  for all  $ \, a \in A_\zero \, $,  $ \, X , Y \in \cL_V(A) \, $,  $ \, A \in \salg^\bullet_\C \; $;
  \item  $ \, \cB \, $  is  $ A_\zero $--antilinear  on the right, i.e.\ it is right-additive and such that   \hfill\break
      \phantom{A} \qquad   $ \; \cB(X,a\,Y) \, = \, \widetilde{a} \, \cB(X,Y) \; $  for all  $ \, a \in A_\zero \, $,  $ \, X , Y \in \cL_V(A) \, $,  $ \, A \in \salg^\bullet_\C \; $;
  \item  $ \cB(X,Y) \, = \,
   \begin{cases}
     \; {(-1)}^{[X]\,[Y]} \, \widetilde{\cB(Y,X)}  &   \text{\ if  $ \, \phi \, $  is  {\sl standard}}  \\
     \;\;\; \widetilde{\cB(Y,X)}  &   \text{\ if  $ \, \phi \, $  is  {\sl graded}}
   \end{cases} $   \hfill\break
 for all  $ \, X , Y \in \cL_V(A) \, $,  $ \, A \in \salg^\bullet_\C \; $.
In short, using notation as in  \eqref{herm-susy}  we can write   \hfill\break
      \phantom{A} \qquad   $ \;\; \cB(X,Y) \, = \, {(-1)}^{(\,\one-\nu_\phi)\,[X]\,[Y]} \, \widetilde{\cB(Y,X)} \;\; $  for any  $ \phi \, $.
%
\end{enumerate}
                                                 \par
   In addition, we say that  $ \cB $  is  {\it consistent\/}  if  $ \, \cB(Y,X) = 0 \, $  for all homogeneous  $ X $,  $ Y $  with different parity, i.e.\  $ \, [X] \not= [Y] \, $.
\end{definition}

\medskip

%
%

\begin{lemma}  \label{B_V ---> cB_L}
 Let  $ \, (V,\phi) \in \smod^\bullet_\C \, $,  and let  $ \, B_V : V \times V \lra \C \, $  be a consistent super Hermitian form on the super vector space  $ V \, $.  Then the natural transformation $ \; \cB_\cL : \cL_V \times \cL_V \relbar\joinrel\relbar\joinrel\lra \cL_\C \; $  defined on objects   --- for  $ \, A \in \salg^\bullet_\C \, $  and homogeneous  $ \, a \in A_{\overline{z}} \, $,  $ \, x \in V_{\overline{z}} \, $,  $ \, b \in A_{\overline{s}} \, $,  $ \, y \in V_{\overline{s}} \, $  ---   by
\begin{equation}  \label{eq: B_V --> cB_L}
  \cB_{\cL_V}\big( a\,x , b\,y \big)  \; := \;  i^{|x| |y|} \, a \, \widetilde{b} \, B_V(x,y)
\end{equation}
%
%
 is a consistent Hermitian form for  $ \cL_V \, $.
\end{lemma}

\begin{proof}  
 The proof is a matter of sheer computation.
%
%
\end{proof}

\smallskip

\begin{proposition}  \label{corr_Herm-x-V_Herm-x-L_V}
 Let  $ \, (V,\phi) \in \smod^\bullet_\C \, $.  Then formula  \eqref{eq: B_V --> cB_L}  realizes a bijection between
 \vskip2pt
   (a) \;  the set of all consistent super Hermitian forms on  $ V $,
 \vskip2pt
   (b) \;  the set of all consistent Hermitian forms for  $ \cL_V \, $.
\end{proposition}

\begin{proof}
 After Lemma \ref{B_V ---> cB_L},  we only need to show that if a form  $ \cB_{\cL_V} $  as in  {\it (b)\/}  is given, then we can find a unique  $ B_V $  on  $ V $  satisfying  \eqref{eq: B_V --> cB_L}.  Indeed, such a  $ B_V $  is defined as follows.  Consider  $ \, A_{\,\xi_+,\xi_-} := \C\big[ \xi_+ , \xi_- \,\big] \in \salg_\C \, $:  this superalgebra has a ``canonical'' standard real structure defined by  $ \, \xi_\pm \mapsto \widetilde{\xi_\pm} := \xi_\mp \, $,  and a ``canonical'' graded one given by  $ \, \xi_\pm \mapsto \widetilde{\xi_\pm} := \pm \xi_\mp \, $.  Then
 also
  $ \,\; \cB_{\cL_V} : \cL_V\big( A_{\,\xi_+,\xi_-} \big) \times \cL_V\big( A_{\,\xi_+,\xi_-} \big) \relbar\joinrel\relbar\joinrel\relbar\joinrel\relbar\joinrel\lra \cL_\C\big( A_{\,\xi_+,\xi_-} \big) \;\, $
 is defined, taking values in  $ \; \cL_\C\big( A_{\,\xi_+,\xi_-} \big) = \C \oplus \C\,\xi_+\,\xi_- \, $,  \,which has  $ \C $--basis  $ \, \{ 1 \, , \xi_+\,\xi_- \} \, $;  thus we can use  \eqref{eq: B_V --> cB_L}  with respect to  $ A_{\,\xi_+,\xi_-} $  to  {\sl define\/} $ B_V $  on  $ V \, $,  and then easily verify that it has all the required properties.
\end{proof}

\smallskip

\begin{definition}  \label{def: cB nondegen/pos-definite}
 We say that an Hermitian form  $ \cB_{\cL_V} $  for  $ \cL_V $  is  {\it non degenerate}, or that it is  {\it positive definite},  if its associated  $ B_V $  is.
\end{definition}

\smallskip

\begin{observation}
 Let  $ \, (V,\phi) \in \smod^\bullet_\C \, $  and let  $ \langle\,\ ,\ \rangle $  be a consistent supersymmetric bilinear form on  $ V $.  Then we can associate to it a natural transformation
\beq  \label{fopts-susyform}
  \langle\,\ ,\ \rangle_{\cL_V} \, : \, \cL_V \times \cL_V \relbar\joinrel\relbar\joinrel\relbar\joinrel\lra \cL_\C  \quad ,   \qquad  {\big\langle a \, x , b \, y \big\rangle}_{\cL_V}  \, := \,  a \, b \, \langle x \, , y \rangle
\eeq
where  $ \, a , b \in A_{\overline{z}} \, $  and  $ \, x , y \in V_{\overline{z}} \, $,  for all  $ \, z \in \Z_2 \, $.  By  Proposition \ref{herm-bil},  there exist two super Hermitian forms  $ B^\pm_V $  associated to $ \, \langle\,\ ,\ \rangle \, $,  and by  Lemma \ref{B_V ---> cB_L}  there exists a unique  $ \cB^\pm_{\cL_V} $  associated to  $ B^\pm_V \, $.  Therefore we can write  $ \cB^\pm_{\cL_V} $  directly from  $ \, \langle\,\ ,\ \rangle_{\cL_V} \, $,  namely (with notation as in  \eqref{herm-susy})
\beq  \label{Bfopts-def}
  \cB^\pm_{\cL_V}(X,Y)  \; = \;  i^{\,(\,1\,\pm\,\nu_\phi)\,[X]\,[Y]} \, {\big\langle X , \varphi_A(Y) \big\rangle}_{\cL_V}   \qquad  \forall \;\; X , Y \in \cL_V(A) \, ,  \; A \in \salg^\bullet_\C
\eeq
%
%
%
%
%
\end{observation}

\subsection{Unitary Lie superalgebras}

 In this section we introduce a general notion of  {\sl unitary superalgebras\/};  in the subsequent subsection then we will also present some relevant examples.

\smallskip

\begin{definition}  \label{def: adjoint}
 Let  $ \, (V,\phi) \in \smod^\bullet_\C \, $,  and let  $ \cB $  be a non-degenerate, consistent Hermitian form on  $ \, \cL_V : \salg_\C^\bullet \longrightarrow \smod_\R \, $.  We define the  {\it adjoint\/} (w.r.t.\  $ \cB \, $)  of  $ \, M \in \big(\text{\sl End}\,(V)\big)(A) \, $  as the unique  $ \, M^\star \in \big(\text{\sl End}\,(V)\big)(A) \, $  defined by
%
%
\beq  \label{adjoint}
  \cB\big( x , M^\star(y) \big)  \; = \;\,
    \begin{cases}
       \; {(-1)}^{[x]\,[M]} \, \cB\big( M(x) \, , y \,\big)  &  \text{\ if  $ \, \bullet = \st \, $}  \\
       \;\;\; \cB\big( M(x) \, , y \,\big)  &   \text{\ if  $ \, \bullet = \gr \, $}
\end{cases}
\eeq
 for all  $ \, x \, , y \in V(A) \, $,  $ \, M \in \big(\text{\sl End}(V)\big)(A) \, $   --- which in the standard case are taken  {\sl homogeneous\/}  with respect to the $ \Z_2 $--grading  whose degree is denoted by  ``$ \, [\ \ ] \, $'',  cf.\ Definition \ref{herm-def}.  Like before, the condition  \eqref{adjoint}  reads  $ \,\; \cB\big( x , M^\star(y) \big) = {(-1)}^{(\,\one\,-\,\nu_\phi)\,[x]\,[M]} \, \cB\big( M(x) \, , y \,\big) \;\, $  with notation as in  \eqref{herm-susy}.
\end{definition}

   The key properties of the adjoint are the following, proved by straightforward check:

\smallskip

\begin{lemma}  \label{adjoint-lemma}
 With notation as in  Definition \ref{def: adjoint}  above, we have
  $$  \displaylines{
   {(a\,M)}^\star  \, = \;  \widetilde{a} \, M^\star  \quad ,  \qquad  M^{\star\star} = \, M  \quad ,   \qquad  {(-M\,)}^\star  \, = \;  -M^\star  \quad ,   \qquad  {(M+N\,)}^\star  \, = \;  M^\star + N^\star  \cr
   {(M\,N\,)}^\star  \, = \;  N^\star \, M^\star  \quad ,   \qquad  {[\,M,N\,]}^\star \, = \; {(-1)}^{(\,\one\,-\,\nu_\phi)\,[M]\,[N]} \big[ N^\star , M^\star \big]  \quad ,   \qquad {\big( J^{-1} \big)}^\star \, = \; {\big( J^\star \big)}^{-1}  }  $$
 for all  $ \, a \in A_\zero \, $,  all  $ \; M , N \in \big(\text{\sl End}\,(V)\big)(A) \; $  and all  $ \, J \in \big(\bGL(V)\big)(A) \; $.
\end{lemma}

\smallskip

\begin{proposition}  \label{adjoint-prop}
 Let the notation be as above, and identify  $ \, \big(\cL_{\fgl(V)}\big)(A) = \big(\text{\sl End}\,(V)\big)(A) \, $.  Then the natural transformation  $ \; \circledast : \cL_{\fgl(V)} \!\relbar\joinrel\lra \cL_{\fgl(V)} \; $  defined on objects by

\beq  \label{eq: adj->real-struct_stand-case}
  M  \; \mapsto \;  M^\circledast := \, \begin{cases}
                         \, -M^\star  &  \;\; \text{if} \;\;\; [M] = \zero \\
                         \, i \, M^\star  &  \;\; \text{if} \;\;\; [M] = \one
   \end{cases}
\qquad   \text{in the  {\sl standard}  case}
\eeq
 and by
\beq  \label{eq: adj->real-struct_grad-case}
  M  \; \mapsto \;  M^\circledast := \, -M^\star  \;\;\;\;\;\; \text{for all} \;\;\; M
\qquad \;\;   \text{in the  {\sl graded}  case}
\eeq
 is a real structure on the functor  $ \cL_{\fgl(V)} \, $,  hence   --- via  Theorem \ref{eq-real} ---   defines a real structure on the complex Lie superalgebra  $ \fgl(V) \, $.
\end{proposition}

\begin{proof}
  By  Definition \ref{gr-str2},  we have to verify properties (1)--(3) therein
and also that  $ \, \circledast \, $  preserves the Lie bracket.  All this follows easily from direct calculations that use  Lemma \ref{adjoint-lemma}.
\end{proof}

\smallskip

\begin{definition}  \label{unit-def}
  Let  $ \cB $  be a non-degenerate, consistent Hermitian form on  $ \cL_V \, $.  We define the  {\it unitary Lie superalgebra\/}  $ \, \fu_\cB(V) \, $  as the functor of fixed points of  $ \; \circledast : \cL_{\fgl(V)} \!\relbar\joinrel\relbar\joinrel\lra \! \cL_{\fgl(V)} \; $,  hereafter denoted  $ \cL_{\fgl(V)}^{\,\circledast} $   --- in the sense of Definition \ref{realf-def}  ---   given on objects by
\beq  \label{unitary}
 \begin{aligned}
   \hskip-25pt   \fu_\cB(V)(A)  \,\; := \;\,  &  \big\{\, M \in \cL_{\fgl(V)}(A) \;\big|\; M^\circledast = \, M \,\big\}  \,\; =  \\
    = \;\;  \Bigg\{\;  &  M \in \cL_{\fgl(V)}(A) \,\;\Bigg|\;\,
 {{\cB\big(M(x) \, , y\big) \, + \, (-1)^{(\one-\nu_\phi)[M][x]} \, \cB\big(x \, , M(y)\big) = \, 0}  \atop  {\forall \;\;\; x \, , y \in V(A)}} \;\Bigg\}
 \end{aligned}
\eeq

\noindent
 Note then that  $ \fu_\cB(V)(A) $  is a  $ \Z_2 $--graded  Lie algebra   --- for all  $ \, A \in \salg_\C^\bullet \, $  ---   just because  $ \cL_{\fgl(V)}(A) $  is
a  $ \Z_2 $--graded  Lie algebra, cf.\  Definition \ref{realf-def}.
\end{definition}

\smallskip

 \begin{remark}
 By  Proposition \ref{gr-rep1},  the functor  $ \fu_\cB(V) $  is always representable.  In addition, in the  {\sl standard\/}  case, i.e.\ when  $ \, (V,\phi) \in \smod^\st_\C \, $,  by  Proposition \ref{re-forms-svs}  the representable functor  $ \fu_\cB(V) $  is represented by the super vector space of all  $ \, \text{\rm m} \in \fgl(V) \, $  such that (see \cite{vsv},  pg.\ 111):
  $$  B\big( \text{\rm m}(x) \,, y \big) \, + \, {(-1)}^{|x||u|} \, B\big( x \,, \text{\rm m}(y) \big)  \; = \; 0  \quad ,  \qquad  \forall \;\; x \, , y \in V  $$
\end{remark}

\smallskip

\begin{observation}  \label{obs: even-part_unitary}
 Let\/  $ \fu_\cB(V) $  be a unitary Lie superalgebra as in  Definition \ref{unit-def};  let also  $ B $  be the consistent super Hermitian form on  $ V $  which corresponds to  $ \cB $  via  Proposition \ref{corr_Herm-x-V_Herm-x-L_V},  which we write as  $ \, B = B_\zero + i\,B_\uno \, $  as in\/  \S \ref{super-Herm-form}.  For each  $ \, \overline{s} \in \Z_2 \, $,  let  $ \, \fu_{B_{\overline{s}}}(V_{\overline{s}}) $  be the classical unitary Lie algebra associated to  $ V_{\overline{s}} $  with the non-degenerate Hermitian form  $ B_{\overline{s}} \, $.  Then the even part of\/  $ \fu_\cB(V) $  is the functor of points of the direct sum Lie algebra  $ \; \fu_{B_\zero}(V_\zero) \oplus \fu_{B_\uno}(V_\uno) \; $.
\end{observation}

\medskip

\subsection{Examples of unitary Lie superalgebras}  \label{sec: ex-unit-supalgs}

 We provide now some examples of real structures, super Hermitian forms and associated unitary Lie superalgebras.

\medskip

\begin{free text}  \label{sbsbsec: Stand_case}
 {\bf Standard real structures on  $ \, \fgl_{m|n} \, $.}
 Let  $ \, V := \C^{m|n} \, $  be endowed   --- like in  Example \ref{unit-ex0}  ---   with the standard real structure  $ \, \phi_\st : \C^{m|n} \!\relbar\joinrel\lra \C^{m|n} \, $  given by  $ \, \phi_\st(z,\zeta) := \big(\,\zbar\,,\zetabar\,\big) \, $,  and the two super Hermitian forms given by
\beq  \label{eq: B_V - st}
  B^\pm_V\big( (z\,,\zeta\,) \,,  \big(z',\zeta'\,\big) \big)  \; :=  \;  z \cdot \overline{z'} \, \pm \, i \, \zeta \!\cdot \overline{\zeta'}
\eeq
   \indent   Following  Lemma \ref{B_V ---> cB_L},  the super Hermitian forms  $ B^\pm_V $  on  $ \, V := \C^{m|n} \, $  correspond to Hermitian forms  $ \cB^\pm_{\cL_V} $  on  $ \, \cL_V \, $,  defined through  \eqref{eq: B_V --> cB_L}:  in detail, these read explicitly
\beq  \label{eq: cB_V}
  \cB^\pm_{\cL_V}\big( (x\,,\xi\,) \,,  \big(x',\xi'\,\big) \big)  \; =  \;  x \cdot \widetilde{x'} \, \mp \, \xi \cdot \widetilde{\xi'}
\eeq
   \indent   Now, according to  Definition \ref{def: adjoint}  we can consider the  {\sl adjoint\/}  of any  $ \, u \in \cL_{\fgl(V)}(A) = \fgl(m|n)(A) \, $  with respect to either  $ \cB^+_{\cL_V} $  or  $ \cB^-_{\cL_V} \, $,  that we will denote by  $ \, u^\star_+ \, $  and  $ \, u^\star_- \, $,  respectively.  After  Proposition \ref{adjoint-prop},  we also have corresponding real structures  $ \circledast_\pm $  on  $ \cL_{\fgl(V)} \, $:  \,in turn, by  Definition \ref{unit-def}  these will define two  {\sl unitary\/}  real forms of $ \cL_{\fgl(V)} \, $,  hence of  $ \, \fgl(V) = \fgl(m|n) \, $  as well. They are given
as follows.
 The explicit form of the adjoint maps is
  $$  u \, = \begin{pmatrix}
                a  &  \beta \;\;  \\
           \gamma  &  d \;\;
             \end{pmatrix}
 \;\; \mapsto \;\;  u^\star_\pm \, =
             \begin{pmatrix}
                a^\star_\pm  &  \beta^\star_\pm \,  \\
           \gamma^\star_\pm  &  d^\star_\pm \,
             \end{pmatrix}
 =
            \begin{pmatrix}
               \, \widetilde{a}^{\,t}  &  \mp\,\widetilde{\gamma}^{\,t} \;\,  \\
        \, \mp\,\widetilde{\beta}^{\,t}  &  \widetilde{d}^{\,t} \;\,
            \end{pmatrix}  $$
from which
we infer the explicit formula of the associated real structures, namely
\beq  \label{expl_real-struct_st_functorial}
 u \, = \begin{pmatrix}
                a  &  \beta \;\;  \\
           \gamma  &  d \;\;
        \end{pmatrix}
 \;\; \mapsto \;\;  u^{\circledast_\pm} \, =
            \begin{pmatrix}
               \, -\,\widetilde{a}^{\,t}  &  \mp \, i \,\widetilde{\gamma}^{\,t} \;\,  \\
        \, \mp \, i \,\widetilde{\beta}^{\,t}  &  -\,\widetilde{d}^{\,t} \;\,
            \end{pmatrix}
\eeq
\noindent
 With these real structures, the associated  {\sl unitary\/}  real forms   --- cf.\  Definition \ref{unit-def}  ---   are
\beq  \label{eq: stand_unit-form_gl(m|n)}
  \fu_{\cB^\pm_V}(V)(A)  \; = \;  \left\{
   \begin{pmatrix}
      a  &  \beta \;\;  \\
 \gamma  &  d \;\;
   \end{pmatrix} \in \fgl(m|n)(A) \;\Bigg|\;
 {{\displaystyle a = -\,\widetilde{a}^{\,t} \; , \;\; \beta = \mp\,i\,\widetilde{\gamma}^{\,t}} \atop {\displaystyle \; \gamma = \mp\,i\,\widetilde{\beta}^{\,t} \; , \;\; d = -\,\widetilde{d}^{\,t}}} \,\right\}
\eeq
                                                                   \par
   Notice that the real structures considered above were defined for the functor of points  $ \cL_{\fgl(V)} \, $.  If instead we look at the Lie superalgebra  $ \, \fgl(V) = \fgl(m|n) \, $  as a superspace, then the real structures  \eqref{expl_real-struct_st_functorial}  on  $ \cL_{\fgl(V)} $ corresponds to the real structures  $ \, \ast_\pm \, $  on  $ \, \fgl(V) = \fgl(m|n) \, $  given by
\beq  \label{expl_real-struct_st_superspace}
   M \, = \begin{pmatrix}
           \text{\rm a}  &  \text{\rm b} \;  \\
           \text{\rm c}  &  \text{\rm d} \;
        \end{pmatrix}
 \;\; \mapsto \;\;  M^{\ast_\pm} \, =
       \begin{pmatrix}
          \, -\,\overline{\text{\rm a}}^{\,t}  &  \mp\,i\,\overline{\text{\rm c}}^{\,t} \;\,  \\
          \, \mp\,i\,\overline{\text{\rm b}}^{\,t}  &  -\,\overline{\text{\rm d}}^{\,t} \;\,
       \end{pmatrix}
\eeq
 which have been previously introduced in  \cite{vsv}, \S 3.4.
%
%

\smallskip

   Similarly, the  {\sl unitary\/}  Lie (sub)superalgebra of  $ \fgl(m|n) $  associated with the real form in  \eqref{eq: stand_unit-form_gl(m|n)},  and representing the functor  $ \fu_{\cB^\pm_V}(V) \, $,  is
  $$  \fu_\cB(m|n)  \; = \;  \left\{
   \begin{pmatrix}
      \text{\rm a}  &  \text{\rm b} \;\;  \\
      \mp\,i\,\overline{\text{\rm b}}^{\,t}  &  \text{\rm d} \;\;
   \end{pmatrix}
       \,\bigg|\;\, \text{\rm a} = -\,\overline{\text{\rm a}}^{\,t} , \, \text{\rm d} = -\,\overline{\text{\rm d}}^{\,t} \,\right\}  $$
\end{free text}

\medskip

\begin{free text}  \label{sbsbsec: Grad_case}

 {\bf Graded real structures on  $ \, \fgl_{m|n} \, $.}  Let again  $ \, V := \C^{m|n} \, $  but consider now its associated functor  $ \cL_V $  as being defined on commutative superalgebras with a  {\sl graded\/}  real structure, hence  $ \, \cL_V : \salggrc \relbar\joinrel\longrightarrow \smod_\C \, $   --- just like in  Definition \ref{herm-def}.  Then we have two natural, consistent, non-degenerate Hermitian forms on  $ \cL_V \, $,  denoted  $ \cB^\pm_\gr \, $,  which are defined on objects by
\beq  \label{eq: cB_V - gr}
  \cB^\pm_\gr\big( (x\,,\xi\,) \,,  \big(x',\xi'\,\big) \big)  \; =  \;  x \cdot \widetilde{x'} \, \mp \, \xi \cdot \widetilde{\xi'}
\eeq
 Note that  \eqref{eq: cB_V - gr}  looks exactly like  \eqref{eq: cB_V}   --- where the functor is defined  $ \salg_\C^\st $  instead.
%
%
%
%
%
%
%
With the same arguments as in  \S \ref{sbsbsec: Stand_case}  above, we find the following explicit form of the adjoint maps
  $$  u \, = \begin{pmatrix}
                a  &  \beta \;\;  \\
           \gamma  &  d \;\;
             \end{pmatrix}
 \;\; \mapsto \;\;  u^\star_\pm \, =
             \begin{pmatrix}
                a^\star_\pm  &  \beta^\star_\pm \,  \\
           \gamma^\star_\pm  &  d^\star_\pm \,
             \end{pmatrix}
 =
            \begin{pmatrix}
               \, \widetilde{a}^{\,t}  &  \pm\widetilde{\gamma}^{\,t} \;\,  \\
        \, \mp\widetilde{\beta}^{\,t}  &  \widetilde{d}^{\,t} \;\,
            \end{pmatrix}  $$
%
 from which we get:
%
\beq  \label{expl_real-struct_st_functorial - BIS}
 u \, = \begin{pmatrix}
                a  &  \beta \;\;  \\
           \gamma  &  d \;\;
        \end{pmatrix}
 \;\; \mapsto \;\;  u^{\circledast_\pm} \, =
            \begin{pmatrix}
               \, -\widetilde{a}^{\,t}  &  \mp\widetilde{\gamma}^{\,t} \;\,  \\
        \, \pm\widetilde{\beta}^{\,t}  &  -\widetilde{d}^{\,t} \;\,
            \end{pmatrix}
\eeq
With these real structures, the associated  {\sl unitary\/}  real forms (via  Definition \ref{unit-def})  are given by
  $$  \fu_{\cB^\pm_V}(V)(A)  \,\; = \;\,  \left\{
   \begin{pmatrix}
      a  &  \beta \;\;  \\
 \gamma  &  d \;\;
   \end{pmatrix} \in \fgl(m|n)(A) \;\Bigg|\; {{\displaystyle a = -\widetilde{a}^{\,t} \; , \;\; \beta = \mp\widetilde{\gamma}^{\,t}} \atop {\displaystyle \; \gamma = \pm\widetilde{\beta}^{\,t} \; , \;\; d = -\widetilde{d}^{\,t}}} \,\right\}  $$
which can be re-written as
  $$  \fu_{\cB^\pm_V}(V)(A)  \,\; = \;\,  \left\{
   \begin{pmatrix}
      a  &  \beta \;\;  \\
 \pm\widetilde{\beta}^{\,t}  &  d \;\;
   \end{pmatrix} \in \fgl(m|n)(A) \;\Bigg|\; {{\displaystyle a = -\widetilde{a}^{\,t}} \atop {\displaystyle d = -\widetilde{d}^{\,t}}} \,\right\}  $$
                                                       \par
   Finally, although we have introduced the real structures directly on the functor  $ \, \cL_{\fgl(V)} \, $,  we can easily see that these structures
%
%
 $ \, \circledast_\pm \, $
 on  $ \cL_{\fgl(V)} $  actually correspond to the real structures  $ \, \ast_\pm : \fgl(V) \relbar\joinrel\relbar\joinrel\lra \fgl(V) \; $  on the Lie superalgebra  $ \, \fgl(V) = \fgl(m|n) \, $  given by
\beq  \label{expl_real-struct_st_Lie-salg - BIS}
 M \, = \begin{pmatrix}
                \text{\rm a}  &  \text{\rm b} \;  \\
                \text{\rm c}  &  \text{\rm d} \;
        \end{pmatrix}
 \;\; \mapsto \;\;  M^{\ast_\pm} \, =
            \begin{pmatrix}
               \, -\overline{\text{\rm a}}^{\,t}  &  \mp\overline{\text{\rm c}}^{\,t} \,  \\
        \, \pm\overline{\text{\rm b}}^{\,t}  &  -\overline{\text{\rm d}}^{\,t} \,
            \end{pmatrix}
\eeq
   In particular,  {\sl  $ \, \ast_+ \, $  has a neat expression in terms of ``supertranspose'' as}  $ \; M^{\ast_+} = -\overline{M}^{\,st} \; $ commonly used by physicists  (see \cite{fl}  and also \cite{serganova,pe}).
%
%
%
\end{free text}

\medskip

\begin{free text}  \label{sbsbsec: Stand+grad_case}
%
%
 {\bf Standard and graded real structures induced by a supersymmetric form.}
 For  $ \, n = 2\,t \, $,  let  $ \, \phi_\st \, $  be the  {\sl standard\/}  real structure on  $ \, V := \C^{m|2t} \, $  considered in  \S \ref{sbsbsec: Stand_case}.  For  $ \, A \in \salg^\st_\C \, $, we write any element of  $ \, V(A) = \C^{m|2t}(A) = A_\zero^{\,m} \times A_\uno^{\,2t} \, $  as a triple  $ \, (x\,,\xi_+\,,\xi_-) \, $  with  $ \, x \in A_\zero^{\,m} \, $  and  $ \, \xi_\pm \in A_\uno^{\,t} \, $. Accordingly, any  $ \, u \in \cL_{\fgl(V)}(A) = \fgl(m|\,2\,t)(A) \, $   --- for  $ \, A \in \salg^\st_\C \, $  ---   will be written as a block matrix
 $ \; u \, = \begin{pmatrix}
                a  &  \beta_+  &  \beta_- \;  \\
      \, \gamma_+  &  d_{\scriptscriptstyle +,+}  &  d_{\scriptscriptstyle +,-} \;  \\
      \, \gamma_-  &  d_{\scriptscriptstyle -,+}  &  d_{\scriptscriptstyle -,-} \;
             \end{pmatrix} \; $
where  $ a $  and  $ d_{\scriptscriptstyle \pm,\pm} $  have entries in  $ A_\zero $  and  $ \beta_\pm $  and  $ \gamma_\pm $  have them in  $ A_\uno \, $;  \,in turn, we will write its adjoint as
 $ \; u^\star_\pm \, = \begin{pmatrix}
            a^\star  &  \beta^\star_+  &  \beta^\star_- \;  \\
  \, \gamma^\star_+  &  d^\star_{\scriptscriptstyle +,+}  &  d^\star_{\scriptscriptstyle +,-} \;  \\
  \, \gamma^\star_-  &  d^\star_{\scriptscriptstyle -,+}  &  d^\star_{\scriptscriptstyle -,-} \;
             \end{pmatrix} \; $.
 \vskip5pt
   With these conventions, the (unique!) Hermitian form  $ \cB^\pm_{\cL_V} $  on  $ \, \cL_V \, $,  that by  Lemma \ref{B_V ---> cB_L}   --- via  \eqref{eq: B_V --> cB_L}  ---   correspond to  $ B^\pm_{\phi_\st} $  on  $ V $  is given explicitly by
\beq  \label{eq: cB_V - st}
  \cB^\pm_{\phi_\st}\big( (x\,,\xi_+\,,\xi_-\,) \,,  \big(x',\xi'_+\,,\xi'_-\,\big) \big)  \; =  \;  x \cdot \widetilde{x'} \, + \, i \, \xi_+ \!\cdot \widetilde{\xi'_-} \, - \, i \, \xi_- \!\cdot \widetilde{\xi'_+}
\eeq
(we still write a superscript  ``$ \, \pm \, $'',  yet it is irrelevant).  Using it, we compute the ``adjoint''  $ \, u^\star := u^\star_\pm \, $  (again unique!) applying the defining conditions  \eqref{adjoint}  to the nine homogeneous summands (that here we read as block-entries) of the matrix
 $ \; u \, =
    \begin{pmatrix}
                a  &  \beta_+  &  \beta_- \;  \\
      \, \gamma_+  &  d_{\scriptscriptstyle +,+}  &  d_{\scriptscriptstyle +,-} \;  \\
      \, \gamma_-  &  d_{\scriptscriptstyle -,+}  &  d_{\scriptscriptstyle -,-} \;
    \end{pmatrix} \; $.
%
%
 The explicit calculations follow again the same arguments as in  \S \ref{sbsbsec: Stand_case}  above; eventually, we find the following explicit form of the adjoint maps
  $$  u
%
%
 \;\; \mapsto \;\;  u^\star \, =
             \begin{pmatrix}
          a^\star  &  \beta^\star_+  &  \beta^\star_- \,  \\
   \gamma^\star_+  &  d^\star_{\scriptscriptstyle +,+}  &  d^\star_{\scriptscriptstyle +,-} \;  \\
   \gamma^\star_-  &  d^\star_{\scriptscriptstyle -,+}  &  d^\star_{\scriptscriptstyle -,-} \;
             \end{pmatrix}
 =
            \begin{pmatrix}
  \widetilde{a}^{\,t}  &  +\,i\,\widetilde{\gamma_-}^{\,t}  &  -\,i\,\widetilde{\gamma_+}^{\,t} \;\;  \\
   -\,i\,\widetilde{\beta_-}^{\,t}  &  +\widetilde{d_{\scriptscriptstyle -,-}}^{\,t}  &  -\widetilde{d_{\scriptscriptstyle +,-}}^{\,t} \;  \\
   +\,i\,\widetilde{\beta_+}^{\,t}  &  -\widetilde{d_{\scriptscriptstyle -,+}}^{\,t}  &  +\widetilde{d_{\scriptscriptstyle +,+}}^{\,t} \;
            \end{pmatrix}  $$
%
%
 and then from the latter we deduce the the associated real structure   --- as in  \eqref{eq: adj->real-struct_stand-case}  ---   namely
\beq  \label{expl_real-struct_st_(sympl)_functorial}
   u \, = \begin{pmatrix}
          a  &  \beta_+  &  \beta_- \;  \\
   \gamma_+  &  d_{\scriptscriptstyle +,+}  &  d_{\scriptscriptstyle +,-} \;  \\
   \gamma_-  &  d_{\scriptscriptstyle -,+}  &  d_{\scriptscriptstyle -,-} \;
             \end{pmatrix}
 \,\; \mapsto \;\;  u^\circledast :=
           \begin{pmatrix}
  -\widetilde{a}^{\,t}  &  -\widetilde{\gamma_-}^{\,t}  &  +\widetilde{\gamma_+}^{\,t} \;  \\
   +\widetilde{\beta_-}^{\,t}  &  -\widetilde{d_{\scriptscriptstyle -,-}}^{\,t}  &  +\widetilde{d_{\scriptscriptstyle +,-}}^{\,t} \;  \\
   -\widetilde{\beta_+}^{\,t}  &  +\widetilde{d_{\scriptscriptstyle -,+}}^{\,t}  &  -\widetilde{d_{\scriptscriptstyle +,+}}^{\,t} \;
           \end{pmatrix}
\eeq
Finally, the  {\sl unitary\/}  real form associated   --- by  Definition \ref{unit-def}  ---   with this real structure is
  $$  \fu_{\cB^\pm_{\phi_\st}}\!(V)(A)  \; = \,  \left\{\! \left.
   \begin{pmatrix}
          a  &  \beta_+  &  \beta_-  \\
   +\widetilde{\beta_-}^{\,t}  &  d_{\scriptscriptstyle +,+}  &  d_{\scriptscriptstyle +,-}  \\
   -\widetilde{\beta_+}^{\,t}  &  d_{\scriptscriptstyle -,+}  &  -\widetilde{d_{\scriptscriptstyle +,+}}^{\,t}
   \end{pmatrix} \! \in \fgl(m|2t)(A) \;\right|\,
{{\displaystyle a \; = \; -\,\widetilde{a}^{\,t}} \atop {\displaystyle d_{\scriptscriptstyle \pm,\mp} = \widetilde{d_{\scriptscriptstyle \pm,\mp}}^{t}}} \,\right\}  $$
                                                                   \par
   Note that the map  $ \; u \mapsto u^\circledast \; $  is the real structure for the functor of points  $ \cL_{\fgl(V)} \, $.  If instead we look at the Lie superalgebra  $ \, \fgl(V) = \fgl(m|\,2\,t) \, $  as a superspace, then the real structure  \eqref{expl_real-struct_st_(sympl)_functorial}  on  $ \cL_{\fgl(V)} $ corresponds to the real structure on  $ \, \fgl(V) = \fgl(m|\,2\,t) \, $  described by
%
%
  $$  M \, = \begin{pmatrix}
          \text{\rm a}  &  \text{\rm b}_+  &  \text{\rm b}_- \;  \\
   \text{\rm c}_+  &  \text{\rm d}_{\scriptscriptstyle +,+}  &  \text{\rm d}_{\scriptscriptstyle +,-} \;  \\
   \text{\rm c}_-  &  \text{\rm d}_{\scriptscriptstyle -,+}  &  \text{\rm d}_{\scriptscriptstyle -,-} \;
             \end{pmatrix}
 \; \mapsto \;  M^\ast :=
            \begin{pmatrix}
  -\overline{\text{\rm a}}^{\,t}  &  -\overline{\text{\rm c}_-}^{\,t}  &  +\overline{\text{\rm c}_+}^{\,t} \;  \\
   +\overline{\text{\rm b}_-}^{\,t}  &  -\overline{\text{\rm d}_{\scriptscriptstyle -,-}}^{\,t}  &  +\overline{\text{\rm d}_{\scriptscriptstyle +,-}}^{\,t} \;  \\
   -\overline{\text{\rm b}_+}^{\,t}  &  +\overline{\text{\rm d}_{\scriptscriptstyle -,+}}^{\,t}  &  -\overline{\text{\rm d}_{\scriptscriptstyle +,+}}^{\,t} \;
            \end{pmatrix}  $$
 Similarly, the  {\sl unitary\/}  Lie (sub)superalgebra of  $ \fgl(m|\,2\,t) $  associated with this real form, and representing the functor  $ \fu_{\cB^\pm_{\phi_\st}}\!(V) \, $,  is
  $$  \fu_{\cB^\pm_{\phi_\st}}\!(m|\,2\,t)  \;\; = \;\;  \left\{ \left.
   \begin{pmatrix}
      \text{\rm a}  &  \text{\rm b}_+  &  \text{\rm b}_- \;  \\
      +\overline{\text{\rm b}_-}^{\,t}  &  \text{\rm d}_{\scriptscriptstyle +,+}  &  \text{\rm d}_{\scriptscriptstyle +,-} \;  \\
      -\overline{\text{\rm b}_+}^{\,t}  &  \text{\rm d}_{\scriptscriptstyle -,+}  &  -\overline{\text{\rm d}_{\scriptscriptstyle +,+}}^{\,t} \;
   \end{pmatrix}
       \,\right|\;\, {{\text{\rm a} \; = \; -\,\overline{\text{\rm a}}^{\,t}_{{}_{\phantom{|}}}} \atop {\text{\rm d}_{\scriptscriptstyle \pm,\mp} = \overline{\text{\rm d}_{\scriptscriptstyle \pm,\mp}}^{\,t}}} \,\right\}  $$
 \vskip7pt
 Finally, a parallel construction starting from the  {\sl graded\/}  real structure  $ \; \phi_\gr : \C^{m|2t} \!\relbar\joinrel\lra \C^{m|2t} \; $  given by  $ \, \phi_\gr(z,\zeta_+,\zeta_-) := \big(\,\zbar\,,+\overline{\zeta_-}\,,-\overline{\zeta_+}\,\big) \, $  provides again, in the first steps, the Hermitian forms  \eqref{eq: cB_V - gr}  of  \S \ref{sbsbsec: Grad_case},  hence the final outcome will be a special instance of what we found therein.
\end{free text}

\medskip

\section{Compact real forms}
%
%

 In this section we describe real forms of basic Lie superalgebras  (see \cite{serganova, parker},  we give a new notion of ``super compactness'', going beyond  \cite{cfv,chuah-mz},  and we describe the associated real structures
%
%
 in the graded and standard case.  We begin with some notation.

\smallskip

\begin{definition}
%
%
 Let  $ V $  be any complex super vector space.  For any  $ \, s \in \{2\,,4\} \, $,  let  $ \, \aut^{\;{}_\R}_{2,s}(V) \, $  be the set of automorphisms  $ \vartheta $  of  $ V $  {\sl as a  {\it real}  vector superspace\/}  such that  $ \; \vartheta\big|_{V_{\overline{z}}} \not= \, \text{\rm id}_{V_{\overline{z}}} \; $  for  $ \, {\overline{z}} \in \Z_2 \, $,  $ \; \vartheta^2\big|_{V_\zero} = \, \text{\rm id}_{V_\zero} \; $  and  $ \; \vartheta^2\big|_{V_\uno} = \, +\text{\rm id}_{V_\zero} \; $  for  $ \, s := 2 \, $  while  $ \; \vartheta^2\big|_{V_\uno} = \, -\text{\rm id}_{V_\zero} \; $  for  $ \, s := 4 \, $.  Then we set:
                                                                            \par
   {\it (a)} \quad  $ \baut_{2,s}(V) \; := \; \big\{\, \theta \in \aut^{\;{}_\R}_{2,s}(V) \,\big|\, \theta \text{\ is\  $ \C $--antilinear} \,\big\} \; $;
                                                                            \par
   {\it (b)} \quad  $ \aut_{2,s}(V) \; := \; \big\{\, \sigma \in \aut^{\;{}_\R}_{2,s}(V) \,\big|\, \sigma \text{\ is\  $ \C $--linear} \,\big\} \; $.
 \vskip3pt
   If in addition  $ \, V = \mathfrak{A} \, $  is a complex associative superalgebra, resp.\ a complex Lie superalgebra, by  $ \, \baut_{2,s}(\mathfrak{A}) \, $ and  $ \, \aut_{2,s}(\mathfrak{A}) \, $  we mean the similar objects defined as above but starting from the set  $ \, \aut^{\;{}_\R}_{2,s}(\mathfrak{A}) \, $  of automorphisms of  $ \mathfrak{A} $  as a  {\it real\/}  (associative, resp.\ Lie) superalgebra with the extra conditions specified above.
\end{definition}

   After  Definition \ref{super-str1},  the elements of  $ \baut_{2,s}(V) $  are exactly the real structures on  $ V \, $;  we will presently show that in special cases these can be classified by the elements of  $ \aut_{2,s}(V) $  too.

\smallskip

\subsection{Real structures of basic (simple) Lie superalgebras}
 Let  $ \fg $  be a complex Lie superalgebra which is  {\sl contragredient},  in the sense of  \cite{kac}, \S 2.5.  Thus  $ \fg $  is defined via a Cartan matrix  $ \, A := {\big( a_{i,j} \big)}_{i,j \in I} \, $   --- with  $ \, I = \{1,\dots,r\} \, $  ---   with entries in  $ \C \, $,  a set of generators  $ \, x^+_i \, $,  $ x^-_i \, $,  $ h_i $  (for all  $ \, i \in I := \{1,\dots,r\} \, $),  of parity  $ \, |h_i| := \zero \, $,  $ \, \big| x^\pm_i \big| := \zero \, $  if  $ \, i \not\in \tau \, $,  $ \, \big| x^\pm_i \big| := \uno \, $  if  $ \, i \in \tau \, $,  for some fixed subset  $ \, \tau \subseteq I \, $.
                                                                    \par
   In addition, we shall say that the set of generators  $ \, {\big\{ x^+_i \, , h_i \, , x^-_i \big\}}_{i \in I} \, $  is  {\it distinguished\/}  if  $ \, |\tau| = 1 \, $   --- in other words,
{\sl there exists one and only one positive simple root which is  {\it odd}}   --- cf.\ \cite{kac, chuah-mz}.

\vskip5pt

\begin{proposition}  \label{omega}
 Let  $ \, \fg $  be contragredient, built out of a Cartan matrix  $ A $  with entries in  $ \mathbb{R} \, $.  Then there exists a unique  $ \, \omega \in \baut_{2,4}(\fg) \, $  such that
  $$  \omega\big(h_j\big) = -h_j  \;\;\; \forall \;\; j \; ,  \qquad  \omega\big(x^\pm_i\big) = -x^\mp_i  \;\;\; \forall \;\; i \not\in \tau \; ,  \qquad  \omega\big(x^\pm_i\big) = \pm x^\mp_i  \;\;\; \forall \;\; i \in \tau $$
\end{proposition}

\begin{proof}
 This is the ``antilinear counterpart'' of a well-known result which guarantees the existence and uniqueness of a  $ \C $--linear  automorphism  $ \omega' $  of  $ \fg $  whose action on the generators is the same as  $ \omega $'s.  One proves it along the same lines as in  \cite{musson},  Proposition 5.1.3 and 5.2.1.
\end{proof}

\smallskip

   Note that when  $ \fg $  is a semisimple Lie algebra, then  $ \omega_\zero $  is the classical Cartan involution corresponding to its compact form  (see \cite{kn}, VI, \S 1).

\smallskip

   From now on,  {\sl we assume our complex Lie superalgebra\/  $ \fg $  to be simple of  {\it basic\/}  type\/},  hence   --- according to the classification  $ \fg $  is of one of the following types:
\beq  \label{list}
   A(m|n) \; ,  \,\;\;  B(m|m) \; ,  \,\;\;  C(n) \; ,  \,\;\;  D(m|n) \; ,  \,\;\;
D(2,1;a) \; ,  \,\;\;  G(3) \; ,  \,\;\;  F(4)
\eeq
 \textit{Moreover, for type  $ D(2,1;a) $  we assume that  $ \, a \in \mathbb{R} \, $}.  In particular, our  $ \fg $  is  {\sl contragredient},  and  Proposition \ref{omega}  above applies.

\smallskip

   We shall now collect a few technical results that we need later.

\medskip

%
%

\begin{lemma}
 Let  $ \fg $  be simple of basic type as in  \eqref{list}  above, with  $ \, a \in \mathbb{R} \, $  for type  $ D(2,1;a) \, $.  Then there exists a suitable positive system  $ \Delta^+ $  and suitable root vectors  $ x_{\pm\alpha} \, \big (\, \pm\alpha \in \Delta^\pm = \Delta_\zero^\pm \cup \Delta_\uno^\pm \,\big) \, $  for which  $ \omega $  as in  Proposition \ref{omega}  gives
  $$  \omega\big(x_{\pm\alpha}\big) \, = \, -x_{\mp\alpha}  \;\;\quad  \forall \;\; \alpha \text{\ even simple} \; ,  \quad \qquad
  \omega\big(x_{\pm\alpha}\big) \, = \, \pm x_{\mp\alpha}  \;\;\quad  \forall \;\; \alpha \text{\ odd simple}  \; . $$
\end{lemma}

\begin{proof}
 Indeed, for  $ \fg $  as in the claim it is known that we can select a distinguished Dynkin diagram, as in  \cite{kac}, p.\ 56, Table VI.  Accordingly, we have unique associated sets of simple roots, of simple root vectors, and of positive/negative roots, as well as a unique associated Cartan matrix.  Then the claim follows as a special instance of  Proposition \ref{omega}.
\end{proof}

\smallskip

\begin{lemma}  \label{parker0}
 Any inner automorphism  $ \phi_\zero  $  of\/  $ \fg_\zero $  can be extended to an inner automorphism  $ \phi $  of\/  $ \fg $   --- i.e.\ one of the form  $ \, \phi = \exp\big(\ad(n)\big) \, $  with  $ \, n \in \fg_\zero \, $.
\end{lemma}

\begin{proof}
 This is proved, in the  {\sl standard\/}  case, by Proposition 2.1 in  \cite{parker}.  In short, given  $ \, \phi_\zero = \exp\big(\ad(n)\big) \, $  on  $ \fg_\zero $  (with  $ \, n \in \fg_\zero \, $),  we can take $ \, \phi_\uno := \exp\big(\ad_\uno(n)\big) \, $,  where  $ \ad_\uno $  denotes the adjoint action of $ \fg_\zero \, $  on  $ \fg_\uno \, $.  In addition, by a straightforward analysis one checks that the very same method actually applies to the  {\sl graded\/}  case as well.
\end{proof}

   The previous lemma has an immediate consequence, whose proof is straightforward.

\smallskip

\begin{lemma}  \label{parker1}
 Let  $ \, \sigma \in \baut_{2,s}(\fg) \, $   --- or  $ \, \sigma \in \aut_{2,s}(\fg) \, $  ---   $ \, s \in \{2\,,4\} \, $,  and let  $ \, \sigma'_\zero \in \baut_2(\fg_\zero) \, $   --- or $ \, \sigma'_\zero \in \aut_2(\fg_\zero) \, $,  respectively.  If  $ \, \sigma'_\zero = \phi_\zero \circ \sigma_\zero \circ \phi_\zero^{-1} \, $  for an inner automorphism  $ \phi_\zero $  of  $ \fg_\zero \, $,  then  $ \sigma'_\zero $  extends to  $ \, \sigma' = \sigma'_\zero + \sigma'_\uno \in \baut_{2,s}(\fg) \, $   --- or to  $ \, \sigma' = \sigma'_\zero + \sigma'_\uno \in \aut_{2,s}(\fg) \, $,  respectively ---   given by  $ \, \sigma' := \phi \circ \sigma \circ \phi^{-1} \, $,  with  $ \, \phi = \phi_\zero + \phi_\uno \, $  as in  Lemma \ref{parker0}  above.
\end{lemma}

\smallskip

   When  $ \; \sigma' = \phi \circ \sigma \circ \phi^{-1} \, $,  \, with  $ \, \sigma \in \baut_{2,s}(\fg) \, $,  $ \, s \in \{2\,,4\} \, $,  for an inner automorphism  $ \phi $,  we will say that  {\sl $ \sigma $  and  $ \sigma' $  are  {\it inner-isomorphic}},  and we will write  $ \, \sigma \simeq \sigma' \, $.

\smallskip

\begin{lemma}  \label{parker2}

 Let  $ \, \sigma, \sigma' \in \baut_{2,s}(\fg) \, $,  $ \, s \in \{2\,,4\} \, $,  with  $ \, \sigma_\zero = \sigma'_\zero \; $.  Then:
 \vskip3pt
   (a) \;  if\/  $ \fg $  is of type 1, then  $ \; \sigma' \, \simeq \, \sigma_\zero \pm \sigma_\uno \; $;
 \vskip3pt
   (b) \;  if\/  $ \fg $  is of type 2, then  $ \; \sigma' \, = \, \sigma_\zero \pm \sigma_\uno \; $.
\end{lemma}

\begin{proof}
 For the standard case, the claim is proved in  Lemma 2.3 and Lemma 2.4 of  \cite{parker}.  The same arguments work in the graded case too.
\end{proof}

\smallskip

\begin{lemma}  \label{cf-1}
 Let  $ \, \theta, \theta' \in \aut_{2,s}(\fg) \, $  with  $ \, s \in \{2\,,4 \} \, $.  If  $ \; \theta_\zero = \theta'_\zero \; $,  then  $ \; \theta_\uno = \pm\theta'_\uno \; $.
\end{lemma}

\begin{proof}
 This is  Proposition 2.3 of \cite{cf3}  for the standard case; the graded case is just an exercise, where one replaces  $ \, \pm\,i \, $  therein with  $ \, \pm\,1 \, $.
\end{proof}

\smallskip

   At last, we have an important result.

\smallskip

\begin{proposition}  \label{aut-prop}
 Let  $ \, \fg $  and  $ \, \omega \in \baut_{2,4}(\fg) \, $  be defined as in  Proposition \ref{omega}.  Then there exist mutually inverse bijections
 \begin{equation}  \label{bij_aut2,4-->baut2,2}
   \begin{aligned}
      &  \aut_{2,4}(\fg) \;\;{\buildrel \Phi_{\scriptscriptstyle \wedge} \over {\lhook\joinrel\relbar\joinrel\relbar\joinrel\relbar\joinrel\relbar\joinrel\relbar\joinrel\twoheadrightarrow}}\;\; \baut_{2,2}(\fg) \,\setminus\, \{\, \theta \;\big|\; \theta{|}_{\fg_\zero} = \omega{|}_{\fg_\zero} \big\} \\
      &  \;\quad  \sigma \,\relbar\joinrel\relbar\joinrel\relbar\joinrel\relbar\joinrel\relbar\joinrel\relbar\joinrel\relbar\joinrel\relbar\joinrel\relbar\joinrel\longrightarrow\, \theta_\sigma := \omega \circ \sigma
   \end{aligned}
 \end{equation}
 and
\begin{equation}  \label{bij_aut2,4<--baut2,2}
   \begin{aligned}
      &  \;\; \aut_{2,4}(\fg) \;\;{\buildrel \Psi_{\scriptscriptstyle \wedge} \over {\twoheadleftarrow\joinrel\relbar\joinrel\relbar\joinrel\relbar\joinrel\relbar\joinrel\relbar\joinrel\rhook}}\;\; \baut_{2,2}(\fg) \,\setminus\, \{\, \theta \;\big|\; \theta{|}_{\fg_\zero} = \omega{|}_{\fg_\zero} \big\}  \\
      &  \omega^{-1} \!\circ \theta =: \sigma_{{}_{\scriptstyle \theta}} \,\longleftarrow\joinrel\relbar\joinrel\relbar\joinrel\relbar\joinrel\relbar\joinrel\relbar\joinrel\relbar\joinrel\relbar\joinrel\relbar\, \theta
   \end{aligned}
 \end{equation}
\noindent
 and also mutually inverse bijections
 \begin{equation}  \label{bij_aut2,2-->baut2,4}
   \begin{aligned}
      &  \aut_{2,2}(\fg) \;\;{\buildrel \Phi_{\scriptscriptstyle \vee} \over {\lhook\joinrel\relbar\joinrel\relbar\joinrel\relbar\joinrel\relbar\joinrel\relbar\joinrel\twoheadrightarrow}}\;\; \baut_{2,4}(\fg) \,\setminus\, \{\, \vartheta \,\big|\, \vartheta{|}_{\fg_\zero} = \omega{|}_{\fg_\zero} \,\vee\; \vartheta{|}_{\fg_\uno} = \omega{|}_{\fg_\uno} \big\}  \\
      &  \;\quad  \varsigma \,\relbar\joinrel\relbar\joinrel\relbar\joinrel\relbar\joinrel\relbar\joinrel\relbar\joinrel\relbar\joinrel\relbar\joinrel\relbar\joinrel\longrightarrow\, \vartheta_\varsigma := \omega \circ \varsigma
   \end{aligned}
\end{equation}
 and
 \begin{equation}  \label{bij_aut2,2<--baut2,4}
   \begin{aligned}
      &  \;\; \aut_{2,2}(\fg) \;\;{\buildrel \Psi_{\scriptscriptstyle \vee} \over {\twoheadleftarrow\joinrel\relbar\joinrel\relbar\joinrel\relbar\joinrel\relbar\joinrel\relbar\joinrel\rhook}}\;\; \baut_{2,4}(\fg) \,\setminus\, \{\, \vartheta \,\big|\, \vartheta{|}_{\fg_\zero} = \omega{|}_{\fg_\zero} \,\vee\; \vartheta{|}_{\fg_\uno} = \omega{|}_{\fg_\uno} \big\}  \\
      &  \omega^{-1} \!\circ \vartheta =: \varsigma_{{}_{\scriptstyle \vartheta}} \,\longleftarrow\joinrel\relbar\joinrel\relbar\joinrel\relbar\joinrel\relbar\joinrel\relbar\joinrel\relbar\joinrel\relbar\joinrel\relbar\, \vartheta
   \end{aligned}
 \end{equation}
\end{proposition}

\begin{proof}
 It is enough to show that the given maps are actually well-defined   --- namely, their images do lie in the expected target set ---   since then they will clearly be pairwise mutually inverse.
                                                     \par
   It follows from definitions and from the explicit description of  $ \, \aut_{2,s}(\fg) \, $   --- see \cite{chuah-mz}  ---   that  $ \omega $  commutes with every  $ \, \eta \in \aut_{2,s}(\fg) \, $  and every  $ \, \gamma \in \baut_{2,s}(\fg) \, $.  This implies (for all  $ \, z \in \Z \, $)
  $$  \displaylines{
   \theta_\sigma^{\,z} = {(\omega \circ \sigma)}^z = \omega^z \circ \sigma^z  \quad ,   \qquad  \sigma_\theta^{\,z} = {\big( \omega^{-1} \!\circ \theta \big)}^z = \omega^{-z} \!\circ \theta^z  \cr
   \vartheta_\varsigma^{\,z} = {(\omega \circ \varsigma)}^z = \omega^z \circ \varsigma^z  \quad , \qquad  \varsigma_{{}_{\scriptstyle \vartheta}}^{\,z} = {\big( \omega^{-1} \!\circ \vartheta \big)}^z = \omega^{-z} \!\circ \vartheta^z  }  $$
Using these identities, a trivial check shows that
  $$  \displaylines{
   \Phi_{\scriptscriptstyle \wedge}\big(\,\aut_{2,4}(\fg)\big) \,\;\subseteq\;\, \baut_{2,2}(\fg) \,\setminus\, \{\, \theta \;\big|\; \theta{|}_{\fg_\zero} = \omega{|}_{\fg_\zero} \big\}
  \quad ,   \qquad  \Psi_{\scriptscriptstyle \wedge}\big(\, \baut_{2,2}(\fg) \setminus \{\, \theta \;\big|\; \theta{|}_{\fg_\zero} \big\} \,\big) \,\;\subseteq\;\, \aut_{2,4}(\fg)  \cr
   \Phi_{\scriptscriptstyle \vee}\big(\,\aut_{2,2}(\fg)\big) \,\;\subseteq\;\, \baut_{2,4}(\fg)  \,\setminus\, \{\, \vartheta \,\big|\, \vartheta{|}_{\fg_\zero} = \omega{|}_{\fg_\zero} \,\vee\; \vartheta{|}_{\fg_\uno} = \omega{|}_{\fg_\uno} \big\}  \cr
   \Psi_{\scriptscriptstyle \vee}\big(\,\baut_{2,4}(\fg)  \,\setminus\, \{\, \vartheta \,\big|\, \vartheta{|}_{\fg_\zero} = \omega{|}_{\fg_\zero} \,\vee\; \vartheta{|}_{\fg_\uno} = \omega{|}_{\fg_\uno} \big\} \,\big) \,\;\subseteq\;\, \aut_{2,2}(\fg)  } $$
 which means that the maps  $ \Phi_{\scriptscriptstyle \wedge} \, $,  $ \Psi_{\scriptscriptstyle \wedge} \, $,  $ \Phi_{\scriptscriptstyle \vee} $  and  $ \Psi_{\scriptscriptstyle \vee} $  are well defined, q.e.d.
\end{proof}

\medskip

\subsection{Compact forms of basic Lie superalgebras}

 We want to define the notion of compact real form of a (complex) Lie superalgebra.  Our notion differs from  \cite{chuah-mz, cf, cfv},  where compact superalgebras are assumed to be even.

\smallskip

\begin{definition}  \label{cpt-def}
 Let  $ \fg $  be a complex Lie superalgebra with a real structure  $ \phi $  on it, and let  $ \cL_\fg^\varphi $   --- see  Definition \ref{realf-def}  ---   be the associated real form (in functorial sense).
 \vskip3pt
    {\it (a)}\;  We say that  $ \cL_\fg^\varphi $  is  {\sl super-compact\/}  if there exists a suitable superspace  $ V $  with a positive definite, consistent Hermitian form  $ \cB $  such that  $ \; \cL_\fg^\varphi \subseteq \fu_\cB(V) \; $   --- cf.\  Definition \ref{unit-def};
 \vskip3pt
    {\it (b)}\;  We say that  $ \cL_\fg^\varphi $  is  {\sl compact\/}  if its even part  $ {\big( \cL_\fg^\varphi \big)}_\zero $  is compact in the classical sense, that is   --- as  $ {\big( \cL_\fg^\varphi \big)}_\zero $  is (always) represented by  $ \, {(\fg_\zero)}^\phi = {\big( \fg^\phi \big)}_\zero \, $  ---   if the real Lie algebra  $ \, {(\fg_\zero)}^\phi = {\big( \fg^\phi \big)}_\zero \, $  is compact in the classical sense.
 \vskip3pt
    {\it (c)}\;  We say that a (graded or standard) real structure  $ \phi $  on  $ \fg $  is  {\sl super-compact},  resp.\ is  {\sl compact},  if the associated real form  $ {\big( \cL_\fg^\varphi \big)}_\zero $  of  $ \cL_\fg^\varphi $  is super-compact, resp.\ is compact.
 \vskip3pt
    {\sl N.B.:}\,  it follows at once from  Observation \ref{obs: even-part_unitary}  that super-compactness implies compactness.
\end{definition}

\smallskip

   The following existence result proves the key importance of the notion of  {\sl graded\/}  real structures. Theorem B in Sec.\ \ref{intro} consists of the statements of  Theorem \ref{cpt-form-gr}  and  Theorem \ref{cpt-form-st}.

\smallskip

\begin{theorem}  \label{cpt-form-gr}
 Let  $ \fg $  be a simple complex Lie superalgebra of basic type, with the additional assumption that  $ \, a \in \mathbb{R} \, $  if  $ \, \fg $  is of type  $ D(2,1;a) \, $.  Then  $ \fg $  has a  {\sl graded},  {\sl super-compact\/}  real structure  $ \omega $  (hence a  {\sl graded},  {\sl super-compact\/}  real form  $ \cL_\fg^\omega \, $)  which is unique up to inner automorphisms.
\end{theorem}

\begin{proof}
 Let  $ \, \omega \in \baut_{2,4}(\fg) \, $  be given as in  Proposition \ref{omega},  and let $ \kappa $  be the Cartan-Killing form of  $ \fg $  (see  \cite{musson}, Sec.\ 5.4, Ch. 5):  by Theorem 5.4.1 in  \cite{musson},  this  $ \kappa $  is non-degenerate.  Using this and other properties of  $ \omega $  and  $ \kappa \, $,  one easily checks that
  $$  \kappa(x,y)  \; = \;  \overline{\kappa\big(\omega(x),\omega(y)\big)}   \quad \qquad  \forall \;\; x , y \in \fg  $$
 Thus the condition of  Proposition \ref{herm-bil}  is satisfied, and therefore
  $$  B(x,y)  \; := \;  {(-i\,)}^{|x| |y|} \, \kappa\big( x \,, \omega(y) \big)  $$
is a consistent super Hermitian form, which in addition is also
positive definite.
                                                       \par
   Now consider the associated Hermitian form (following  Lemma \ref{B_V ---> cB_L})
  $$  \cB\big( a\,x , b\,y \big)  \; = \;  i^{|x| |y|} \, a \, \widetilde{b} \, B(x,y)  $$
for all homogeneous  $ \, a \in A_{\overline{z}} \, $,  $ \, x \in V_{\overline{z}} \, $,  $ \, b \in A_{\overline{s}} \, $,  $ \, y \in V_{\overline{s}} \, $  and all  $ \, A \in \salg^\gr_\C \, $;  more directly, according to  \eqref{Bfopts-def}  we can also write
\beq  \label{Bkappa}
  \cB(X\,,Y)  \; = \;  \kappa_A\big(\, X \,,\, \omega_A(Y) \big)   \quad \qquad  \forall \;\; X , Y \in \cL_\fg(A)
\eeq
   \indent   We want to show that the functor  $ \cL_\fg^\omega $  embeds into  $ \fu_\cB(\fg) \, $:  this is equivalent to showing that
\beq  \label{U_is_B-unitary}
  \cB\big( U \!\cdot\! X \, , \, Y \big) \, + \, \cB\big( X \, , \, U \!\cdot\! Y \big)  \; = \;  0
\eeq
for all  $ \, X , Y \in \cL_\fg(A) \, $  and  $ \, U \in \cL_\fg^\omega(A) \, $,  where  $ \; U \!\cdot\! X := [U,X] \; $.  Note that in the present case the super vector space  $ V $  of  Definition \ref{unit-def}  is just  $ \fg $  itself.  Now, thanks to  \eqref{Bkappa}  we have
%
%
  $$  \cB( U \!\cdot\! X , Y)  =  \kappa_A\big( [U,X] , \omega_A(Y) \big)  =  -\kappa_A\big( X , \big[U,\omega_A(Y)\big] \big)  =  -\kappa_A\big( X , \omega_A[U,Y] \big)  =  -\cB( X , U \!\cdot\! Y )  $$
%
%
 since  $ \kappa $  is  {\sl ad\/}--invariant  and  $ \; \omega_A(U) = U \; $  by assumption; thus  \eqref{U_is_B-unitary}  is proved.
 \vskip7pt
   We now come to uniqueness.  By the ordinary theory, a real structure  $ \, \phi_\zero \, $  on  $ \fg_\zero $  giving a compact real form of the latter is unique up to inner automorphism, i.e.\ we can write any other real structure  $ \, \phi'_\zero \, $  on  $ \fg_\zero $  yielding another compact form as  $ \; \phi'_\zero = \psi_\zero \circ \phi_\zero \circ \psi_\zero^{-1} \; $  for some inner automorphism  $ \psi_\zero \, $.  Thanks to this, if  $ \, \phi' \, $  is any real structure on  $ \fg $  giving a compact form  $ \fk' \, $,  then  Lemma \ref{parker1}  applies and we conclude our proof.
\end{proof}

\smallskip

   We now turn our attention to the standard case.

\smallskip

\begin{theorem}  \label{cpt-form-st}
 Let  $ \fg $  be a simple complex Lie superalgebra of basic type.  Then:
 \vskip3pt
   {\it (a)}\,  if  $ \fg $  is of type 1 (i.e., of type  $ A $  or  $ C \, $),  then it admits a  {\sl standard, compact}  real structure, which is unique up to inner automorphisms;
%
%
 \vskip2pt
   {\it (b)}\,  if  $ \fg $  is of type 2 (i.e., of type  $ B \, $,  $ D \, $,  $ F $  or  $ G \, $),  then it has no  {\sl standard, compact}  real structure.
\end{theorem}

\begin{proof}
 {\it (a)}\,  If  $ \fg $  is of type  $ A \, $,  then  $ \, \fg = \fsl(m\!+\!1|\,n\!+\!1) \, $  or  $ \, \fg = \mathfrak{psl}(m\!+\!1|\,m\!+\!1) \, $.  In both cases, one easily sees that the standard structures in  $ \fgl(m\!+\!1|\,n\!+\!1) $  described in  \S \ref{sbsbsec: Stand_case}  induce similar structures on  $ \fg \, $,  and we are done.  Finally, uniqueness follows as in the proof of  Theorem \ref{cpt-form-gr}.
                                                             \par
   If  $ \fg $  is of type  $ C $  instead, we find an explicit  $ \, \sigma \in \baut_{2,2}(C(n)) \, $  making explicit use of  Proposition \ref{aut-prop},  namely in the form  $ \; \sigma := \omega \circ \theta \, $;  \,here  $ \omega $  is as in  Proposition \ref{omega},  while  $ \, \theta \in \aut_{2,4}\big(C(n)\big) \, $  is chosen to be the identity on  $ {C(n)}_\zero $  and such that  $ \, \theta(X_\beta) := i\,X_\beta \, $  for  $ \beta $  the odd simple root
in a positive system with preferred simple system (i.e., a simple system with one odd root, now denoted  $ \beta $).  Once we describe  $ \fg $  of type  $ C(n) $  as the Lie superalgebra  $ \mathfrak{osp}\big(2\big|2(n\!-\!1)\big) \, $   --- see  \cite{kac}, p.\ 31 ---   a straightforward analysis yields the following explicit description of  $ \sigma $
\begin{equation}  \label{eq: sigma_x_C(n)}
  \sigma\begin{pmatrix}
    \; b  &  0  &  x  &  y \;\;  \\
    \; 0  & -b  &  z  &  w \;\;  \\
   \; w^t & y^t &  A  &  B \;\;  \\
  \; -z^t & -x^t & C  &  D \;\;
        \end{pmatrix}
\; = \;
    \begin{pmatrix}
      \; -\bbar  &     0     & -i\,\wbar  &  i\,\zbar \;\;  \\
      \;    0    &   \bbar   &  i\,\ybar  & -i\,\xbar \;\;  \\
  \; -i\,\xbar^t & i\,\zbar^t & -\Abar^t  &   -\Cbar \;\;   \\
  \; -i\,\ybar^t & i\,\wbar^t & -\Bbar^t  &    \Abar \;\;
    \end{pmatrix}
\end{equation}
 --- where the above are block matrices with blocks of convenient sizes ---   from which one can directly check that actually  $ \, \sigma \in \baut_{2,2}(C(n)) \, $,  \,as required.
                                                            \par
   As to uniqueness, it follows again as in the proof of  Theorem \ref{cpt-form-gr}.
 \vskip5pt
   {\it (b)}\,  In this case, the statement is discussed in detail in  \cite{chuah-jalg},  where the condition of admissible marking   --- see (1.4) in  \cite{chuah-jalg}  ---   prescribes one even root to be non compact. For the reader convenience we briefly recap here the argument. According to  Theorem \ref{aut-prop}  a real form corresponds to an automorphism  $ \, \theta \in \aut_{2,4}(\fg) \, $,  assigning the eigenvalue  $ \, i \, $  to  $ \, x_\beta \in \fg_\beta \, $,  with  $ \beta $  a simple odd root in the simple system as in  Proposition \ref{omega}.  Since  the lowest root  $ \; \varphi = 2\,\be + \dots \; $  is even, the eigenvalue of  $ \varphi $  is  $ \, -1 \, $, hence  $ \varphi $  is non compact.  Hence  $ \fg_0 $  is non compact, consequently we cannot have a standard compact real form for  $ \fg \, $ (see also  \cite{chuah-jalg} Secc.\ 1, 2).
\end{proof}

\smallskip

\begin{remark}
 In particular for  $ \, \fg = \mathfrak{osp}\big(2\,\big|\,2(n-1)\big) \, $   --- i.e., of type $ C(n) $  ---   one easily sees that the real form defined by the standard real structure  $ \sigma $  in  \eqref{eq: sigma_x_C(n)}   --- i.e., the real Lie subsuperalgebra of fixed points of  $ \sigma $  in  $ \fg $  ---   is given by
  $$  {\mathfrak{osp}\big(2\,\big|\,2(n\!-\!1)\big)}^\sigma  \; = \;  \left\{
   \begin{pmatrix}
       \, i\,b  &    0   &      x    &     y     \,  \\
       \,   0   &  -i\,b &  i\,\ybar & -i\,\xbar \,  \\
 \, -i\,\xbar^t &   y^t  &      A    &     B     \,  \\
 \, -i\,\ybar^t  & -x^t  &   -\Bbar  &   \Abar \,
   \end{pmatrix}  \;\; : \;\;  b \in \R \; ,  \;\;  A = -\Abar^t \;\; ,  \;\;  B = B^t \;\right\} $$
\end{remark}

\medskip

\subsection{Cartan involutions and decomposition}  \label{cartan-dec}

 If  $ \fg_\zero $  is a complex semisimple Lie algebra, we have a one to one correspondence between non compact real forms of  $ \fg_\zero $  and involutions  $ \theta_\zero $  of  $ \fg_\zero \, $. Now  $ \theta_\zero $  restricts to a Cartan involution on the corresponding real form, unique up
to inner automorphism.  We wish to extend this picture to the graded setting.

\medskip

   Let  $ \fg $  be a contragredient basic Lie superalgebra,  $ \fh $  a Cartan subalgebra and let  $ \, \theta \in \aut_{2,2}(\fg) \, $  be  {\it equal rank},  that is  $ \, \theta{\big|}_\fh = \id_\fh \, $.  As in  Proposition \ref{aut-prop},  we have that  $ \, \sigma = \omega \circ \theta \in \baut_{2,4}(\fg) \, $  gives a graded real structure on  $ \fg \, $.  Let  $ \, \fk = \fg^\theta \, $. Since  $ \theta $  commutes with  $ \omega \, $,  we have that  $ \theta $  preserves this structure, hence  $ \, \big(\, \fk \, , \sigma{\big|}_\fk \,\big) \, $  is a well defined graded real structure on $ \fk \, $.

\smallskip

\begin{proposition}
 Let the notation be as above. Then  $ \, \big(\, \fk \, , \sigma{\big|}_\fk \,\big) \, $  is super-compact.
\end{proposition}

\begin{proof}
 By the arguments of  Theorem \ref{cpt-form-gr},  we immediately see that  $ \, \cL^\sigma_\fk \subseteq \fu(\kappa) \, $.
\end{proof}

\smallskip

   Let  $ \fp $  be the eigenspace of  $ \theta $  of eigenvalue  $ -1 \, $.  Then we immediately have the decomposition:
\beq  \label{st-cartan-dec}
  \fg  \; = \;  \fk \oplus \fp  \;\; ,  \qquad  [\fk,\fk] \subseteq \fk \; ,  \qquad  [\fk,\fp] \subseteq \fp
\eeq

   \indent   This complex decomposition is preserved by the graded real structure  $ \sigma $  and then we shall call it the  {\it Cartan decomposition\/}  of the graded real form  $ (\fg,\sigma) \, $.  Notice that, by  Lemmas \ref{parker1} and \ref{cf-1},  the Cartan automorphism  $ \theta $  and the corresponding Cartan decomposition are unique up to inner automorphism.

\medskip

   We now turn to examine the standard case  (see \cite{chuah-mz}).  When  $ \fg $  is of type 2, the lack of compact forms  (see Theorem \ref{cpt-form-st})  makes the case  $ \, \fk = \fk_\zero \, $ studied in  \cite{chuah-mz, cfv}  most relevant.  We invite the reader to consult those references for more details.  So we focus on the case when  $ \fg $  is of type 1.
                                                         \par
   Let  $ \, \theta \in \aut_{2,4}(\fg) \, $  be an equal rank automorphism.  Let  $ \fk_\zero \, $, $ \fp_\zero $  be the eigenspaces of eigenvalues  $ \pm 1 $  for  $ \theta_\zero \, $,  let $ \Delta^k_\zero $  the root system of the semisimple part of  $ \fk_\zero \, $.  Choose a distinguished simple system, that is one with only one odd root  $ \be \, $.  Define
  $$  \fk  \; := \;  \fk_\zero \oplus {\textstyle \sum\limits_{\alpha \in \Delta^k_\zero}} \fg_{\pm(\beta+\alpha)} \;\; , \qquad   \fp  \; := \;  \fp_\zero \oplus {\textstyle \sum\limits_{\alpha \not\in \Delta^k_\zero}} \fg_{\pm(\beta+\alpha)}  $$
 Then, most immediately we have the decomposition as above:
\beq  \label{gr-cartan-dec}
  \fg  \; = \;  \fk \oplus \fp  \;\; ,  \qquad  [\fk,\fk] \subseteq \fk \; ,  \qquad  [\fk,\fp] \subseteq \fp
\eeq

   \indent   An easy check shows that it is preserved by the standard real structure  $ \sigma $ associated with  $ \theta \, $,  hence we call it the  {\it Cartan decomposition\/}  of the standard real form  $ (\fg,\sigma) \, $.  As before, we notice that by  Lemmas \ref{parker1} and \ref{cf-1}  the Cartan automorphism  $ \theta $  and the corresponding Cartan decomposition are unique up to inner automorphism.

\medskip

\section{Real forms of basic supergroups}

 In this section, we shall provide a global version of the infinitesimal real forms constructed in the previous sections.

\subsection{Unitary supergroups}  \label{sbsec: unitary-spgrps}

 Let  $ \, (V,\phi) \in \smod_\C^\bullet \, $  be a complex super vector space with (standard or graded) real structure, and  $ B $  a consistent, non-degenerate, positive definite super Hermitian
form on it.  Proposition \ref{adjoint-prop}  provides a real structure  $ \; \circledast : \cL_{\fgl(V)} \!\relbar\joinrel\relbar\joinrel\lra \! \cL_{\fgl(V)} \; $  on  $ \cL_{\fgl(V)} \, $,  which corresponds to a real structure on the Lie superalgebra  $ \fgl(V) \, $.  By  Proposition \ref{prop:eq_defs-rst_sgrps-sHCp's},  there exists a unique real structure
  $$  \circledast^{\scriptscriptstyle \bG} : \bGL^\bullet(V) \relbar\joinrel\relbar\joinrel\relbar\joinrel\longrightarrow \overline{\bGL^\bullet}(V)  $$
on the supergroup  $ \bGL(V) $ corresponding to it.  In particular, on an element  $ \, g = g_+ \cdot \exp(\cY) \in \big(\bGL(V)\big)(A) \, $   --- as in  \eqref{eq: g(elem)-descr_via-sHCp}  ---   using the exponential notation, we have
\begin{equation}  \label{eq: real-str_sgrp-element}
  g^{\circledast^{\scriptscriptstyle \bG}}  \; = \;\;  g_+^{\circledast^{\scriptscriptstyle G}_\zero} \!\cdot \exp\!\big(\, {\cY}^{\,\circledast_{\!A}} \big)
\end{equation}
where  $ \circledast^{\scriptscriptstyle G}_\zero $  is the ordinary real structure  on $ \, \rGL(V_\zero) \times \rGL(V_\uno) \, $,  namely
  $$  g_+^{\circledast^G_\zero}  \; = \;\,
   \begin{pmatrix}
      \, a  &  0 \;\,  \\
      \, 0  &  d \;\,
   \end{pmatrix}^{\!\circledast^G_\zero}  \; = \;\,
  \left(\! \begin{pmatrix}
        \, a  &  0 \;\,  \\
        \, 0  &  d \;\,
           \end{pmatrix}^{\!\star} \;\right)^{\!\!-1}  \; = \;\,
%
%
    \Bigg( \begin{matrix}
       {\big( a^{-1} \big)}^{\!\star}  &  \!\!\!\!\!\!\! 0 \,  \\
          0  &  \!\!\!\!\! {\big( d^{-1} \big)}^{\!\star} \,
            \end{matrix} \Bigg)  $$
while, by  Proposition \ref{adjoint-prop}  and  Lemma \ref{adjoint-lemma},
  $$  {\cY}^{\,\circledast_{\!A}}  \; = \;  \left(\, \sum\limits_{n=1}^{+\infty}\, {{\;{(-1)}^{n-1}\;} \over {\;n\;}} \, \cX^{\,n} \,\right)^{\!\circledast}  \, = \;
%
%
 \varepsilon \, \sum\limits_{n=1}^{+\infty}\, {{\;{(-1)}^{n-1}\;} \over {\;n\;}} \, {\big( \cX^{\,\star} \big)}^n  \, = \;  \varepsilon \, \log\big(\,\mathbf{1}+\cX^{\,\star}\,\big)  $$
 --- where  $ \, \varepsilon := i \, $  or  $ \, \varepsilon := -1 \, $  according to whether we are in the standard or the graded case.  So
  $$  \exp\big( {\cY}^{\,\circledast_{\!A}} \big) = \, \exp\big(\, \varepsilon \, \big( \log\big( \mathbf{1} + \cX^\star \big) \big)  \; = \;  {\big( \mathbf{1} + \cX^{\,\star} \,\big)}^\varepsilon  \; = \;  {\textstyle\sum\limits_{n=0}^N} \, \bigg({\varepsilon \atop n}\bigg) {\big(\cX^{\,\star}\big)}^n  $$
 where  $ N $  is the least non-negative integer such that  $ \; {\big(\cX^\star\big)}^{N+1} = \, 0 \, \in \, \big(\text{\sl End}(V)\big)(A) \; $.  Therefore
\begin{equation}  \label{eq: g = block matrix - explicit}
  g^{\circledast^\bG}  = \;  \begin{pmatrix}
             \, a  & \beta \;\,  \\
         \, \gamma &   d \;\,
              \end{pmatrix}^{\!\circledast^\bG}
 \!\! = \;\,  \Bigg( \begin{matrix}
          {\big( a^\star \big)}^{\!-1}  &  \!\!\!\!\!\!\! 0 \,  \\
             0  &  \!\!\!\!\! {\big( d^\star \big)}^{\!-1} \,
                     \end{matrix} \Bigg)
   \cdot  \left( \begin{pmatrix}
        \, 1  &  \!\! a^{-1} \beta \;\,  \\
    \, d^{-1} \gamma \!\!  &   1 \;\,
                 \end{pmatrix}^{\!\!\star} \;\right)^{\!\varepsilon}
\end{equation}
 {\sl Note\/}  in addition that the  {\sl graded\/}  case   --- when  $ \, \varepsilon = -1 \, $  ---   also reads
\begin{equation}  \label{eq: g = block matrix - GRAD-explicit}
  g^{\circledast^\bG}  \, = \;\,  \begin{pmatrix}
             \, a  & \beta \;\,  \\
         \, \gamma &   d \;\,
              \end{pmatrix}^{\!\circledast^\bG}
    \!\! = \;  \left( \begin{pmatrix}
           \, a  &  \beta \;\,  \\
       \, \gamma  &   d \;\,
                    \end{pmatrix}^{\!\!\star} \,\right)^{\!-1}
       \!\! = \;  \left( \begin{pmatrix}
           \, a  &  \beta \;\,  \\
       \, \gamma  &   d \;\,
                    \end{pmatrix}^{\!\!-1} \,\right)^{\!\star}   \quad \qquad  \text{({\sl graded\/}  case)}
\end{equation}

\smallskip

\begin{definition}  \label{def: unit_grps}
 We define the  {\it unitary supergroup\/}  $ \, \bU_\cB(V) \, $,  with respect to the super Hermitian form $ \cB \, $,  as the real form of  $ \bGL(V) $  corresponding to the real structure $ \circledast^{\scriptscriptstyle \bG} \, $.  Explicitly, it is
  $$  \bU_\cB(V)(A)  \,\; := \;\,  \Big\{\; g \in \big(\bGL(V)\big)(A) \;\Big|\; g^{\circledast^\bG} = \, g \;\Big\}   \eqno \forall \;\; A \in \salg_\C^\bullet  \qquad  $$
\end{definition}
   {\sl N.B.:}\,  It follows at once from  Observation \ref{obs: even-part_unitary}  that the even part of a unitary supergroup is the direct product of two ordinary unitary groups.

\smallskip

\begin{examples}
   {\it (a)}\,  Let  $ \, V := \C^{1|1} \, $  with the  {\sl standard\/}  real structure given in \S \ref{sbsbsec: Stand_case}.  Then the associated standard real structure  $ \, \circledast^{\scriptscriptstyle G} \, $  on the supergroup  $ \, \bGL(V) = \bGL_{1|1} \, $  is given explicitly as follows
(see also \cite{fi1}):
  $$  \begin{pmatrix}
             \, a  & \beta \;\,  \\
         \, \gamma &   d \;\,
              \end{pmatrix}^{\!\circledast^\bG}
 \, = \;\;
   \begin{pmatrix}
      \; \widetilde{a}^{-1} \big(\, 1 + \widetilde{a}^{-1} \widetilde{\beta} \, \widetilde{d}^{-1} \, \widetilde{\gamma} \,\big)   &   \mp \, i \, \widetilde{a}^{-1} \widetilde{d}^{-1} \, \widetilde{\gamma} \;\,  \\
      \; \mp \, i \, \widetilde{d}^{-1} \, \widetilde{a}^{-1} \widetilde{\beta}   & \widetilde{d}^{-1} \big(\, 1 + \widetilde{d}^{-1} \, \widetilde{\gamma} \, \widetilde{a}^{-1} \widetilde{\beta} \,\big) \;
   \end{pmatrix}  $$
 \vskip5pt
   {\it (b)}\,  Let  $ \, V := \C^{m|n} \, $  with the  {\sl graded\/}  real structure given in \S \ref{sbsbsec: Grad_case}.  Then the associated graded real structure  $ \, \circledast^{\scriptscriptstyle G} \, $  on the supergroup  $ \, \bGL(V) = \bGL_{m|n} \, $  is given explicitly as follows:
  $$  \begin{pmatrix}
             \, a  & \beta \;\,  \\
         \, \gamma &   d \;\,
              \end{pmatrix}^{\!\circledast^\bG}
 \, = \;\;
   \begin{pmatrix}
      \; \widetilde{a}^{\,t}   &   \pm \, \widetilde{\gamma}^{\,t} \;\,  \\
      \; \mp \, \widetilde{\beta}^{\,t}   &   \widetilde{d}^{\,t} \;\,
   \end{pmatrix}^{\!-1}  $$
\end{examples}

\medskip

\subsection{Compact real forms of supergroups}

 Our notion of compact supergroup will be modelled on the one of Lie superalgebras  (cf.\  Definition \ref{cpt-def}),  therefore, it is stronger than the one commonly seen in the literature, which amounts to
``topological compactness'' only (see \cite{cfv, fi1}).

\smallskip

\begin{definition}  \label{cpt-def_sgrps}
 Let  $ \bG $  be a complex Lie supergroups with a real structure  $ \Phi $  on it, and let  $ \bG^\Phi $   --- see  Definition \ref{def:re-form_sgrps_1}  ---   be the associated real form.
 \vskip3pt
    {\it (a)}\;  We say that  $ \bG^\Phi $  is  {\sl super-compact\/}  if there exists a suitable superspace  $ V $  with a non-degenerate, positive definite, consistent Hermitian form  $ \cB $
such that  $ \; \bG^\Phi \leq \bU_\cB(V) \; $  (see Definition \ref{def: unit_grps}).
 \vskip3pt
    {\it (b)}\;  We say that  $ \bG^\Phi $  is  {\sl compact\/}  if its even part  $ {\big( \bG^\Phi \big)}_\zero $  is compact in the classical sense.
 \vskip3pt
    {\it (c)}\;  We say that a (graded or standard) real structure  $ \Phi $  on  $ \bG $  is  {\sl super-compact},  resp.\ is  {\sl compact},  if the associated real form  $ \bG^\Phi $  is super-compact, resp.\ is compact.
 \vskip3pt
    {\it N.B.:}\;  it is immediate to see that super-compactness implies compactness.
\end{definition}

   Let  $ \bG $  be a complex supergroup, with tangent Lie superalgebra  $ \, \fg := \Lie(\bG) \, $.  We say that  $ \bG $  is  {\sl basic\/}  if  $ \fg $  is simple of basic type.
 \vskip3pt
   Now assume that a complex supergroup  $ \bG $  is  {\sl connected and simply connected}.  Then, it is clear by  \S \ref{subsec: real-structs-sgrps}  that any real structure on  $ \fg $  integrates to a real structure (of the same order) on  $ \bG \, $.  In particular, if  $ \bG $  is also basic, we have the following, direct consequence of  Theorem \ref{cpt-form-gr}:

\smallskip

\begin{theorem}  \label{cpt-form-gr_sgrps}
 Let  $ \bG $  be a connected, simply connected, basic, complex supergroup, with  $ \, a \in \R \, $  if  $ \bG $  is of type  $ D(2,1;a) \, $.  Then  $ \bG $  admits a  {\sl graded},  {\sl super-compact}  real structure\/  $ \Omega $   --- hence a  {\sl graded},  {\sl super-compact\/}  real form  $ \bG^\Omega $  ---   which is unique up to inner automorphisms, whose associated real structure on  $ \, \fg := \Lie(\bG) \, $  is the real structure  $ \omega $  of  Theorem \ref{cpt-form-gr}.
\end{theorem}

\smallskip

   Similarly, we have the following, straightforward consequence of  Theorem \ref{cpt-form-st}:

\smallskip

\begin{theorem}  \label{cpt-form-st_sgrps}
 Let  $ \bG $  be a connected, simply connected, basic, complex supergroup.  Then:
 \vskip3pt
   {\it (a)}\;  if\/  $ \Lie(\bG) $  is of type 1 (i.e., of type  $ A $  or  $ C \, $),  then  $ \bG $  admits a {\sl standard},  {\sl compact}  real structure, which is unique up to inner automorphisms;
%
%
 \vskip2pt
   {\it (b)}\;  if\/  $ \Lie(\bG) $  is of type 2 (i.e., of type  $ B $,  $ D $,  $ F $  or  $ G \, $),  then  $ \bG $  has no  {\sl standard},  {\sl compact}  real structure.
\end{theorem}

\smallskip

\begin{observation}
 We can also immediately construct the real forms associated with the real structures  $ (\fg,\sigma) $  of Sec.\ \ref{cartan-dec}.  It is not difficult to see that we have the standard and graded global Cartan decompositions associated to the Cartan decompositions  \eqref{gr-cartan-dec}  and \eqref{st-cartan-dec},  that is
  $$  \bG  \; \cong \;  \mathbf{K} \cdot \mathbf{P}  $$
 where  $ \bK $  is the supergroup associated with the superalgebra  $ \, \fk = \fg^\sigma \, $  and  $ \; \bP \, \cong \, \bP_\zero \times \A^{0|d_1}_{\bullet} \; $,  while  $ \bP_\zero $  is the space appearing in the ordinary global Cartan decomposition  (see \cite{kn}, Ch.\ VI).  Clearly on  $ \bG $ we have the real structure induced by  $ \sigma \, $,  which restricts also to  $ \bK $  and to $ \bP \, $.
\end{observation}

\end{document}